\documentclass[AmsartStyle]{arxram-aif}
\usepackage{pictex}

\newcommand\geo[1]{\overline{\rule[0pt]{0pt}{7pt} #1}}
\equalenv{pro}{prop}
\equalenv{lem}{lemm}
\equalenv{cor}{coro}

\equalenv{dfn}{defi}
\equalenv{exa}{exam}
\equalenv{rmk}{rema}
\equalenv{rmks}{remas}

\begin{document}
\title[Harmonic functions on treebolic space]{Brownian motion on treebolic space:\\
positive harmonic functions}
\alttitle{Mouvement Brownien sur l'espace arbolique: fonctions harmoniques positives}

\author[A. Bendikov]{\firstname{Alexander} \lastname{Bendikov}}
\address{Insitute of Mathematics, Wroclaw University\\
Pl. Grundwaldzki 2/4, 50-384 Wroclaw, Poland}
\email{bendikov@math.uni.wroc.pl}

\author[L. Saloff-Coste]{\firstname{Laurent} \lastname{Saloff-Coste}}
\address{Department of Mathematics, Cornell University\\ 
Ithaca, NY 14853, USA}
\email{lsc@paris.math.cornell.edu}

\author[M. Salvatori]{\firstname{Maura} \lastname{Salvatori}}
\address{Dipartimento di Matematica, Universit\`a di Milano\\
Via Saldini 50, 20133 Milano, Italy}
\email{mauras@mat.unimi.it}

\author[W. Woess]{\firstname{Wolfgang} \lastname{Woess}}
\address{Institut f\"ur Diskrete Mathematik\\
Technische Universit\"at Graz\\
Steyrergasse 30, A-8010 Graz, Austria}
\email{woess@TUGraz.at}

\thanks{Partially supported by Austrian Science Fund (FWF) projects 
P24028 and W1230 and Polish NCN grant DEC-2012/05/B/ST 1/00613}

\keywords{Tree, hyperbolic plane, horocyclic product, quantum complex,
Laplacian, positive harmonic functions}

\altkeywords{Arbre, plan hyperbolique, produit horocyclique, complexe quantique,
Laplacien, fonctions harmoniques positives}

\subjclass
                 {31C05; 
                  60J50, 
                  53C23, 
                  05C05} 

\begin{abstract} This paper studies                                   
potential theory on treebolic space, that is, the horocyclic product
of a regular tree and hyperbolic upper half plane.
Relying on the analysis on strip complexes developed 
by the authors, a family of 
Laplacians with ``vertical drift'' parameters is considered.
We investigate the positive harmonic
functions associated with those Laplacians.  
\end{abstract}

\begin{altabstract}
Ce travail est dedi\'e a une etude de la th\'eorie du potentiel sur l'espace 
arbolique, i.e., le produit horcyclique d'un \^arbre r\'egulier avec
le demi plan hyperbolique superieur.  En se basant sur l'analyse sur les complexes
a bandes Riemanniennes develop\'ee par les auteurs, on consid\`ere une famille de
Laplaciens avec deux param\`etres concernant la d\'erive verticale. On examine les fonctions
harmoniques associ\'ees a ces Laplaciens.
\end{altabstract}

\maketitle

\section{Introduction}\label{sec:intro}
We recall the basic description of treebolic space. For many further
details on the geometry, metric structure and isometry group, the reader is
referred to \cite{BSSW2}. 
 
$$
\beginpicture 

\setcoordinatesystem units <1.1mm,1.1mm>

\setplotarea x from -16 to 54, y from 0 to 48

\plot 0 0  40 20  48 36  8 16  0 0  -8 16  -2 28  38 48  36.666 45.333 /

\plot 48 36  54 48  14 28  8 16  2  28  42 48   44.4 43.2 /

\plot 4.8 22.4  -8 16  -14 28  26 48  28.4  43.2 /

\endpicture
$$
\begin{center}
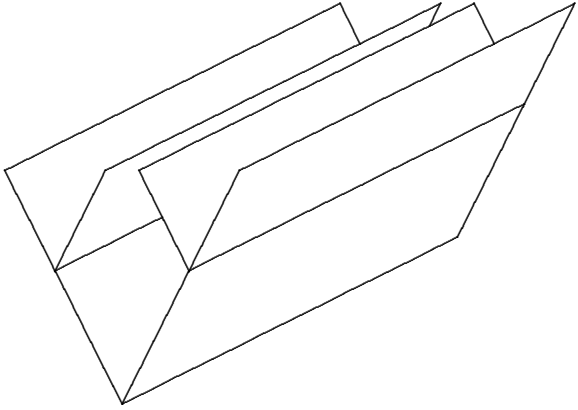
{\sl Figure 1.\/} A compact piece of treebolic space, with $\mathsf{p}=2$.\footnote{This 
figure also appears in \cite{BSSW1}.}
\end{center}

\medskip

We consider the homogeneous tree 
${\mathbb T}={\mathbb T}_{\mathsf{p}}\,$, drawn in such a way that each vertex $v$ has
one predecessor $v^-$ and $\mathsf{p}$ successors. 
We consider ${\mathbb T}$ as a metric graph, where each edge is a copy of the
unit interval $[0\,,\,1]$. The discrete, integer-valued graph metric 
$\mathsf{d}_{{\mathbb T}}$ on the vertex set ($0$-skeleton) $V({\mathbb T})$ of ${\mathbb T}$ has an obvious 
``linear'' extension to the entire metric graph.
We can partition the vertex set into countably many sets $H_k\,$, $k \in {\mathbb Z}$, such
that each $H_k$ is countably infinite, and every vertex $v \in H_k$ has
its predecessor $v^-$ in $H_{k-1}$ and its successors in $H_{k+1}$. We write $[v^-,v]$
for the metric edge between those two vertices, parametrised by the unit interval. For any
$w \in [v^-,v]$ we let $t = k-\mathsf{d}_{{\mathbb T}}(w,v)$ and set $\mathfrak{h}(w)=t$. In particular, 
$\mathfrak{h}(v) = k$ for $v \in H_k\,$. In general, we set $H_t = \{ w \in {\mathbb T} : \mathfrak{h}(w)=t \}$.
These sets are the \emph{horocycles.} 
We choose a root vertex $o \in H_0\,$.

Second, we consider hyperbolic upper half space ${\mathbb H}$, and draw the
horizontal lines $\mathsf{L}_k = \{ x+\mathfrak{i}\, \mathsf{q}^k : x \in \mathbb{R} \}$, thereby subdividing
${\mathbb H}$ into the strips 
$\mathsf{S}_k = \{ x+\mathfrak{i}\, y : x \in \mathbb{R}\,,\; \mathsf{q}^{k-1} \le y \le \mathsf{q}^{k} \}$, 
where $k \in {\mathbb Z}$. We sometimes write ${\mathbb H}(\mathsf{q})$ for the resulting strip complex, which
is of course hyperbolic plane in every geometric aspect.

The reader is invited to have a look at the respective figures of ${\mathbb T}_{\mathsf{p}}$ and ``sliced''
hyperbolic plane ${\mathbb H}(\mathsf{q})$ in \cite{BSSW2}. Treebolic space with parameters $\mathsf{q}$ and
$\mathsf{p}$ is then
\begin{equation}\label{eq:treebolicdef}
{\mathsf{HT}}(\mathsf{q},\mathsf{p}) = \{ \mathfrak{z} = (z,w) \in {\mathbb H} \times {\mathbb T}_{\mathsf{p}} : 
\mathfrak{h}(w) = \log_{\mathsf{q}}(\operatorname{\text{\sl \!Im}} z) \}\,.
\end{equation}

Thus, in treebolic space ${\mathsf{HT}}(\mathsf{q},\mathsf{p})$, infinitely many copies of the strips
$\mathsf{S}_k$ are glued together as follows: to each vertex $v$ of ${\mathbb T}$ there corresponds 
the \emph{bifurcation line}
$$
\mathsf{L}_v = \{ (x + \mathfrak{i}\, \mathsf{q}^k, v) : x \in  \mathbb{R} \} = \mathsf{L}_k \times \{v\}\,,
\quad \text{where} \quad \mathfrak{h}(v) = k\,.
$$
Along this line, the strips
$$
\mathsf{S}_v = \bigl\{ (x+\mathfrak{i}\, y,w) : x \in \mathbb{R}\,,\; 
y \in [\mathsf{q}^{k-1} \,,\,\mathsf{q}^{k}]\,,\; w \in [v^-,v]\,,\;
\mathfrak{h}(w) = \log_{\mathsf{q}}y \bigr\}
$$
and $\mathsf{S}_u\,$, where $u \in V({\mathbb T})$ with $u^-=v$, are glued together.
The \emph{origin} of our space is $\mathfrak{o} = (0,o)$.

\smallskip

In \cite{BSSW2}, we outlined several viewpoints why this space is interesting.
It can be considered as a concrete example of a Riemannian complex in the spirit of 
{\sc Eells and Fuglede}~\cite{EeFu}.
Treebolic space is a \emph{strip complex} (``quantum complex'')
in the sense of \cite{BSSW1}, where a careful study
of Laplace operators on such spaces is undertaken, taking into account the
serious subtleties arising from the singularites of such a complex along its
\emph{bifurcation manifolds}. In our case, the latter are the lines $\mathsf{L}_v\,$,
$v \in V({\mathbb T})$. The other interesting feature is that treebolic spaces are 
\emph{horocyclic products} of ${\mathbb H}(\mathsf{q})$ and ${\mathbb T}_{\mathsf{p}}\,$, with a structure
that shares many features with the  \emph{Diestel-Leader graphs}  $\mathsf{DL}(\mathsf{p},\mathsf{q})$
\cite{DiLe} as well as with the manifold (Lie group) $\text{\sf Sol}(\mathsf{p},\mathsf{q})$, 
the horocyclic product of two hyperbolic planes
with curvatures $-\mathsf{p}^2$ and $-\mathsf{q}^2$, respectively, where $\mathsf{p}, \mathsf{q} > 0$. 
Compare with {\sc Woess}~\cite{Wlamp}, {\sc Brofferio and Woess}~\cite{BrWo1}, 
\cite{BrWo2} and with 
{\sc Brofferio, Salvatori and Woess}~\cite{BSW}; see also
the survey by {\sc Woess}~\cite{Wo-IMN}. The graph $\mathsf{DL}(\mathsf{p},\mathsf{p})$ is a 
Cayley graph of the Lamplighter group ${\mathbb Z}_{\mathsf{p}}\wr {\mathbb Z}$. Similarly, 
the Baumslag-Solitar group
$\operatorname{\sf BS}(\mathsf{p}) = \langle a,b \mid ab = b^{\mathsf{p}}a \rangle$ is a prominent
group that acts isometrically and with compact quotient on ${\mathsf{HT}}(\mathsf{p},\mathsf{p})$. 

\smallskip

In the present paper, we take up the thread from \cite{BSSW2}, where we
have investigated the details of the metric structure, geometry and
isometries before describing the spatial asymptotic behaviour 
of Brownian motion on treebolic space. Brownian motion is induced
by a Laplace operator $\Delta_{\alpha,\beta}\,$, where the parameter
$\alpha$ is the coefficient of a vertical drift term in the interior of each
strip, while $\beta$ is responsible for (again vertical) Kirchhoff type
bifurcation conditions along the lines $\mathsf{L}_v\,$.  
The serious task of rigorously constructing  $\Delta_{\alpha,\beta}$
as an essentially self-adjoint diffusion operator was undertaken in
the general setting of strip complexes in \cite{BSSW1}. For the
specific case of ${\mathsf{HT}}$, it is explained in detail in \cite{BSSW2}.
We shall recall only its basic features, while we shall
use freely the geometric details from \cite{BSSW2}.

\smallskip

Here, we face the rather difficult issue to describe, resp. determine 
the positive harmonic functions on ${\mathsf{HT}}(\mathsf{q},\mathsf{p})$. 
In \S 2, we recall the basic 
features of treebolic space and the family of Laplacians with Kirchhoff conditions
at the bifurcation lines. In \S 3, we start with a Poisson representation
on ``rectangular'' sets which are compact.  Then we obtain such a representation
on simply connected sets which are unions of strips and a solution
of the Dirichlet problem on those sets. (For the technical details of the main results
presented here, the reader is referred to the respective sections.) For the following, 
let $\overline \Omega$ be the union of all closed strips adjactent to $\mathsf{L}_o$
and $\Omega$ its interior, and let $\mu_{\mathfrak{w}}^{\Omega}$ be the exit distribution from
$\Omega$ of Brownian motion starting at $\mathfrak{w} \in \Omega$.
\\[8pt]
\textsc{Theorems \ref{thm:poiss} \& \ref{thm:dir}.} ---
{\sl 
\emph{(I)} Let $h$ be continuous on $\overline \Omega$ and harmonic on $\Omega$, and
suppose that on that set, $h$ grows at most exponentially with respect to the metric
of ${\mathsf{HT}}$. Then for every $\mathfrak{w} \in \Omega$,
$$
h(\mathfrak{w}) = \int_{\partial \Omega} h(\mathfrak{z}) \, 
d\mu_{\mathfrak{w}}^{\Omega}(\mathfrak{z})\,.
$$
\emph{(II)} Let $f$ be continuous on $\partial \Omega$, and
suppose that on that set, $f$ grows at most exponentially with respect to the metric
of ${\mathsf{HT}}$. Then 
$$   
h(\mathfrak{w}) = 
\int_{\partial \Omega} f(\mathfrak{z}) \, d\mu_{\mathfrak{w}}^{\Omega}(\mathfrak{z})\,, 
\quad \mathfrak{w} \in \Omega 
$$    
defines an extension of $f$ that is harmonic on $\Omega$ and continuous 
on $\overline\Omega$. 
}
\\

The difficulties arise from the singularities at the bifurcation lines plus the
fact that $\Omega$ is unbounded in the horizontal direction. 
We obtain that 
the law $\mu = \mu_{\mathfrak{w}}^{\Omega}$
of the random walk induced by Brownian motion on the union $\operatorname{\sf LT}$
of all bifurcation lines has a continuous density and exponential tails.
The main result of \S 4 is as follows.
\\[8pt]
\textsc{Theorem \ref{thm:restrict-HT}.} ---
{\sl Restricting positive 
harmonic functions on ${\mathsf{HT}}$ to $\operatorname{\sf LT}$ yields a one-to-one relation with the positive 
harmonic functions of the random walk driven by $\mu$ on $\operatorname{\sf LT}$.
}
\\

This is of interest, because the disconnected subspace $\operatorname{\sf LT} \subset {\mathsf{HT}}$ 
and that random walk are invariant under the
transitive action of the isometry group of ${\mathsf{HT}}$.
In \S 5, we use Martin boundary theory to prove the following decomposition.
\\[8pt]
\textsc{Theorem \ref{thm:decompose}.} ---
{\sl Every positive harmonic function $h$ on ${\mathsf{HT}}$ has the form
$$
h(\mathfrak{z}) = h^{{\mathbb H}}(z) + h^{{\mathbb T}}(w)\,, \quad \mathfrak{z} = (z,w) \in {\mathsf{HT}}\,,
$$
where $h^{{\mathbb H}}$ is non-negative harmonic on ${\mathbb H}(\mathsf{q})$ and
$h^{{\mathbb T}}$ is non-negative harmonic on ${\mathbb T}_{\mathsf{p}}$.   
}
\\

Given the drift parameters $\alpha$ and $\beta$, there is a formula for the 
\emph{vertical drift}  $\ell(\alpha,\beta) \in \mathbb{R}$, and $|\ell(\alpha,\beta)|$ is the rate of 
escape of our Brownian motion.
\\[8pt]
\textsc{Theorem \ref{thm:Liouville-HT}.} ---
{\sl Suppose that $\mathsf{p} \ge 2$. Then the Laplacian $\Delta_{\alpha,\beta}$ 
on ${\mathsf{HT}}$ has the weak Liouville property
(all bounded harmonic functions are constant) if and only if 
$\ell(\alpha,\beta)=0$.
}
\\

Finally, in the situation when the projection of the Laplacian
on the hyperbolic plane is smooth, we can describe all minimal harmonic functions
explicitly. 

\smallskip

The techniques that we employ here can be applied to obtain
similar results on other types of strip complexes. The simplest one
is where one replaces ``sliced'' hyperbolic by Euclidean plane. Very close to the present
study is the case where one takes hyperbolic upper half space ${\mathbb H}_d$ in arbitrary 
dimension, subdivided by level sets (horospheres) of the Busemann function with respect to
the boundary point $\boldsymbol{\infty}$. Also of interest: to replace the tree by the $1$-skeleton
of a higher-dimensional $\widetilde A_d$-building, or just to take different types of level
functions on a regular tree; compare e.g. with the space considered by 
{\sc Cuno and Sava-Huss}~\cite{CS}.

\section{Laplacians on treebolic space}\label{sec:Laplacians}

We briefly recapitulate the most important features of ${\mathsf{HT}}$, in particular,
the construction of
our family of Laplace operators $\Delta_{\alpha,\beta}$ on
${\mathsf{HT}}(\mathsf{q},\mathsf{p})$ with 
parameters $\alpha \in \mathbb{R}$ and $\beta > 0$.
For more details, see \cite{BSSW2}. 
The rigorous construction is carried out in \cite{BSSW1}.

For any function $f: {\mathsf{HT}} \to \mathbb{R}$, we write $f_v$ for its restriction to $\mathsf{S}_v\,$.
For $(z,w) \in \mathsf{S}_v\,$, the element $w \in [v^-,v]$ is uniquely 
determined,
so that we can omit $w$ and write $f_v(z) = f(z,w)$: formally, $f_v$ is defined
on $\mathsf{S}_k \subset {\mathbb H}$, where $k = \mathfrak{h}(v)$. 

\begin{dfn}\label{def:Cinfty} We let  ${\mathcal C}^\infty({\mathsf{HT}})$ be the space of 
those continuous functions $f$ on ${\mathsf{HT}}$ such that, for each $v\in V({\mathbb T})$, the 
restriction $f_v$ on $\mathsf{S}_v$  has continous derivatives 
$\partial_x^m\partial_y^n f_v(z)$ of all orders 
in the interior $\mathsf{S}^o_v$ which satisfy, for all $R>0$,
$$
\sup\bigl\{|\partial_x^m\partial_y^n f_v(z)|:
|\operatorname{\text{\sl \!Re}} z| \le R\,,\; \mathsf{q}^{\mathfrak{h}(v)-1} < \operatorname{\text{\sl \!Im}} z < \mathsf{q}^{\mathfrak{h}(v)} \bigr\}
<\infty\,.
$$
\end{dfn}

Thus, on each strip $\mathsf{S}_v\,$, each partial derivative has a continuous
extension $\partial_x^m\partial_y^n f_v(z)$  
to the strip's boundary. However, except for $m=n=0$, when $w^-=v$, we do in general \emph{not} have 
that  $\partial_x^m\partial_y^n f_w = \partial_x^m\partial_y^n f_v$
on $\mathsf{L}_v = \mathsf{S}_v \cap \mathsf{S}_w\,$. The hyperbolic gradient
$$
\nabla f_v(z) = \bigl(y^2 \partial_x f_v(z)\,,\,y^2 \partial_y f_v(z)\bigr)
$$ 
is defined in the interior of each strip.
On any bifurcation line $\mathsf{L}_v\,$, we have to distinguish
between all the one-sided limits of the gradient, obtaining the family
$$
\nabla f_v(z) \quad\text{and}\quad \nabla f_w(z) \quad \text{for all}\; w \in V({\mathbb T})\;
\text{with}\; w^-=v\,,\quad (z,v) \in \mathsf{L}_v\,.  
$$ 
For any open domain $\Theta \subset {\mathsf{HT}}$, we let let ${\mathcal C}^\infty_c(\Theta)$
be the space of all functions in ${\mathcal C}^\infty({\mathsf{HT}})$ which have compact 
support contained in $\Theta\,$. 
We let
\begin{equation}\label{eq:LT}
\operatorname{\sf LT} = \bigcup_{v \in V({\mathbb T})} \mathsf{L}_v \quad\text{and}\quad {\mathsf{HT}}^o =  \bigcup_{v \in V({\mathbb T})} \mathsf{S}_v^o = 
{\mathsf{HT}} \setminus \operatorname{\sf LT}\,.
\end{equation}
The \emph{area element} of ${\mathsf{HT}}$ is $d\mathfrak{z} = y^{-2} dx\,dy$ for $\mathfrak{z}=(z,w) \in {\mathsf{HT}}^o\,$,
where $z=x+\mathfrak{i}\, y$ and $dx$, $dy$ are Lebesgue measure: this is (a copy of) the 
hyperbolic upper half plane area element. The area of the lines $\mathsf{L}_v$ is $0$. 
For $\alpha\in \mathbb{R}\,$, $\beta>0$, we define the measure $\mathbf{m}_{\alpha,\beta}$
on ${\mathsf{HT}}$ by
\begin{equation}\label{eq:mab}
\begin{aligned}
d\mathbf{m}_{\alpha,\beta}(\mathfrak{z})&= \phi_{\alpha,\beta}(\mathfrak{z})\,d\mathfrak{z}\,,
\quad\text{with}\\
\phi_{\alpha,\beta}(\mathfrak{z}) &= \beta^{\mathfrak{h}(v)}\,y^\alpha \quad \text{for}\quad
\mathfrak{z}=(x+\mathfrak{i}\, y,w)\in \mathsf{S}_v \setminus \mathsf{L}_{v^-}\,,
\end{aligned}
\end{equation}
where $v \in V({\mathbb T})$, that is, $w \in (v^-\,,v]$ and $\log_{\mathsf{q}} y = \mathfrak{h}(w)$.
\begin{dfn}\label{def:ADHT}   For $f\in {\mathcal C}^\infty({\mathsf{HT}})$
and $\mathfrak{z}=(x+\mathfrak{i}\, y,w)\in {\mathsf{HT}}^o\,$, 
$$
\Delta_{\alpha,\beta}f(\mathfrak{z})= y^{2}(\partial_x^2+\partial_y^2)f(\mathfrak{z})
+\alpha \,y\, \partial_yf(\mathfrak{z})\,.
$$
Let $\mathcal{D}^\infty_{\alpha,\beta,c}$ be the
space of all functions $f \in {\mathcal C}^\infty_c({\mathsf{HT}})$ with the following
properties. 
\\[5pt]
(i) For any $k$,  the  $k$-th iterate $\Delta_{\alpha,\beta}^k f$, originally 
defined on ${\mathsf{HT}}^o,$  admits a continuous extension to all of ${\mathsf{HT}}$ (which then
belongs to ${\mathcal C}^\infty_c({\mathsf{HT}})$ and is also denoted $\Delta_{\alpha,\beta}^k f$).
\\[5pt]
(ii) The function $f$, as well as each of its iterates $\Delta_{\alpha,\beta}^k f$,
satisfies the \emph{bifurcation conditions}
\begin{equation}\label{eq:bif}
\partial_y f_v=\beta\sum_{w\,:\,w^-=v}
\partial_y f_w \quad
\text{on $L_v$ for each $v \in V({\mathbb T})\,$.}
\end{equation}
\end{dfn}

\begin{pro}\label{pro:saHT}\cite{BSSW1}
The space $\mathcal D^\infty_{\alpha,\beta,c}$ is dense in the Hilbert space
${\mathcal L}^2({\mathsf{HT}},\mathbf{m}_{\alpha,\beta})$. 

The operator $(\Delta_{\alpha,\beta},\mathcal D^\infty_{\alpha,\beta,c})$ 
is essentially self-adjoint in ${\mathcal L}^2({\mathsf{HT}},\mathbf{m}_{\alpha,\beta})$. 
\end{pro}
We also write $\bigl(\Delta_{\alpha,\beta},\text{\rm Dom}(\Delta_{\alpha,\beta})\bigr)$ 
for its unique self-adjoint extension. Basic properties of this Laplacian
and the associated heat semigroup are derived in \cite{BSSW1}.
In particular, there is the positive, continuous, symmetric heat kernel
$\mathbf{h}_{\alpha,\beta}(t,\mathfrak{w},\mathfrak{z})$ on $(0,\infty)\times {\mathsf{HT}}\times{\mathsf{HT}}$
such that for all $f\in {\mathcal C}_c({\mathsf{HT}})$,
$$
e^{t\Delta_{\alpha,\beta}}f(\mathfrak{z})=\int_{{\mathsf{HT}}}
\mathbf{h}_{\alpha,\beta}(t,\mathfrak{w},\mathfrak{z})\,f(\mathfrak{z})\,d\mathbf{m}_{\alpha,\beta}(\mathfrak{z})\,.
$$
$\Delta_{\alpha,\beta}$ is the infinitesimal generator of our Brownian motion $(X_t)_{t \ge 0}$
on ${\mathsf{HT}}$. It has infinite life time and continuous sample paths.
For every starting point $\mathfrak{w}\in{\mathsf{HT}}$, its distribution $\mathbb P^{\alpha,\beta}_\mathfrak{w}$
on $\boldsymbol{\Omega}=\mathcal C ([0,\infty]\to {\mathsf{HT}})$ is determined by
the one-dimensional distribution
$$
\mathbb P^{\alpha,\beta}_{\mathfrak{w}}[X_t\in U]
= \int_U \mathbf{h}_{\alpha,\beta}(t,\mathfrak{w},\mathfrak{z}) \,d\mathbf{m}_{\alpha,\beta}(\mathfrak{z})
= \int_U \mathbf{p}_{\alpha,\beta}(t,\mathfrak{w},\mathfrak{z}) \,d\mathfrak{z}\,,
$$
where $U$ is any Borel subset of ${\mathsf{HT}}$ and
\begin{equation}\label{eq:pab}
\mathbf{p}_{\alpha,\beta}(t,\mathfrak{w},\mathfrak{z}) = 
\mathbf{h}_{\alpha,\beta}(t,\mathfrak{w},\mathfrak{z})\,\phi_{\alpha,\beta}(\mathfrak{z})
\end{equation}
with $\phi_{\alpha,\beta}$ as in \eqref{eq:mab}. 

\smallskip

We have the projections $\pi^{{\mathbb H}}: {\mathsf{HT}} \to {\mathbb H}\,$, $\pi^{{\mathbb T}}: {\mathsf{HT}} \to {\mathbb T}$
and $\pi^{\mathbb{R}}: {\mathsf{HT}} \to \mathbb{R}\,$, where for $\mathfrak{z} = (z,w)$, 
$$
\pi^{{\mathbb H}}(\mathfrak{z}) = z\,, \quad  \pi^{{\mathbb T}}(z,w) = w\,, \quad\text{and}\quad
\pi^{\mathbb{R}}(\mathfrak{z}) = \log_{\mathsf{q}} \operatorname{\text{\sl \!Im}}(z) = \mathfrak{h}(w)\,. 
$$
On several occasions it will be useful to write
$$
\operatorname{\text{\sl \!Re}} \mathfrak{z} = \operatorname{\text{\sl \!Re}} z \quad \text{for}\quad \mathfrak{z} = (z,w) \in {\mathsf{HT}}\,.
$$

The ``sliced'' hyperbolic ${\mathbb H}(\mathsf{q})$ plane is 
the treebolic space 
${\mathsf{HT}}(\mathsf{q},1)$: the tree is the bi-infinite line graph.
In particular, we have the operator $\Delta^{{\mathbb H}}_{\alpha,\beta}$ on ${\mathbb H}$. 
We also have a Laplacian $\Delta^{{\mathbb T}}_{\alpha,\beta}$ on the metric
tree, whose rigourous construction (in the same way as above) is considerably
simpler. Note the different parametrisation in ${\mathbb T}$, where each edge $[v^-,v]$
corresponds to the real interval $[\mathfrak{h}(v)-1\,,\mathfrak{h}(v)]$. 
We write $f_v$ for the restriction of $f: {\mathbb T} \to \mathbb{R}$ to $[v^-,v]$, which depends on
one real variable. 
The analogue of the measure of \eqref{eq:mab} is 
\begin{equation}\label{eq:mabT}
\begin{aligned}
d\mathbf{m}_{\alpha,\beta}^{{\mathbb T}}(w)&= \phi_{\alpha,\beta}^{{\mathbb T}}(w)\,dw\,,\quad\text{with}\\
\phi_{\alpha,\beta}^{{\mathbb T}}(w) &= \beta^{\mathfrak{h}(v)}\,\mathsf{q}^{(\alpha-1)\mathfrak{h}(w)}\,\log\mathsf{q}
\quad \text{for}\quad w \in (v^-,v]\,,
\end{aligned}
\end{equation}
where $v \in V({\mathbb T})$ and $dw$ is
the standard Lebesgue measure in each edge. The space ${\mathcal C}^\infty({\mathbb T})$
is as in Definition \ref{def:Cinfty}, with the edges of
${\mathbb T}$ in the place of the strips. 
Definition \ref{def:ADHT} and the bifurcation condition 
\eqref{eq:bif} are replaced by the following:
every $f \in \text{\rm Dom}(\Delta^{{\mathbb T}}_{\alpha,\beta})
\cap {\mathcal C}^\infty({\mathbb T})$ must satisfy for every $v \in V({\mathbb T})$ 
\begin{equation}\label{eq:bif-LapT}
\begin{aligned} f_v'(v)&=\beta\sum_{w\,:\,  w^- = v} f_w'(v)\quad\text{and}\quad\\
\Delta^{{\mathbb T}}_{\alpha,\beta}f 
&=\frac{1}{(\log \mathsf{q})^2} f'' +  \frac{\alpha-1}{\log \mathsf{q}} f'\quad 
\text{in the open edge}\;(v^-,v)\,.
\end{aligned}
\end{equation}
We also have this on the real line, by identifying
$\mathbb{R}$ with the metric tree with vertex set ${\mathbb Z}$ and degree 2. 
The edges are the intervals $[k-1\,,\,k]$, $k \in {\mathbb Z}$, and the Laplacian is 
$\Delta^{\mathbb{R}}_{\alpha,\beta}\,$. Its definition
as a differential operator in each open interval $(k-1\,,\,k)$ is  
as in \eqref{eq:bif-LapT}, and the bifurcation condition now is 
$f'(k-) = \beta\, f'(k+)$ for all $k \in {\mathbb Z}$. We have the following.

\begin{pro}\label{pro:projections} Let $(X_t)$ be Brownian motion on  ${\mathsf{HT}}(\mathsf{q},\mathsf{p})$ 
with infinitesimal generator $\Delta_{\alpha,\beta}\,.$\\[4pt]
\emph{(a)}  $\; Z_t =\pi^{{\mathbb H}}(X_t)$ is the Markov process on 
${\mathbb H}(\mathsf{q})$ with 
generator $\Delta^{{\mathbb H}}_{\alpha,\beta\mathsf{p}}$. 
\\[4pt]
\emph{(b)} $\;W_t=\pi^{{\mathbb T}}(X_t)$ is the Markov process on ${\mathbb T}_{\mathsf{p}}$ with 
generator $\Delta^{{\mathbb T}}_{\alpha,\beta}$. 
\\[4pt]
\emph{(c)} $\;Y_t=\pi^{\mathbb{R}}(X_t)$ is the Markov process on on $\mathbb{R}$ with 
generator $\Delta^{\mathbb{R}}_{\alpha,\beta\mathsf{p}}$. 
\end{pro}

We also need to recall further features of the geometry of ${\mathsf{HT}}$, as well as
of ${\mathbb T}$ and ${\mathbb H}$. For two points $w, w' \in {\mathbb T}$, their \emph{confluent} $w \curlywedge w'$ 
is the unique element $v$ on the geodesic path  $\geo{w\,w'}$ where $\mathfrak{h}(v)$ is minimal.
In the specific case when $w$ and $w'$ lie on a geodesic ray spanned by a sequence of vertices that
are successive predecessors, the confluent is is one of $w$ and $w'$, while otherwise it is always a 
vertex. Analogously, for two points $z, z' \in {\mathbb H}$, we let $z \wedge z'$
be the point on the hyperbolic geodesic $\geo{z\,z'}$ where the imaginary part is maximal.
Recall that hyperbolic geodesics lie on semi-circles which are orthogonal to the bottom
boundary line $\mathbb{R}$, resp. on vertical straight lines. It is a straightforward (Euclidean) exercise 
to see that 
\begin{equation}\label{eq:RE-IM}
|\operatorname{\text{\sl \!Re}} z - \operatorname{\text{\sl \!Re}} z'| \le 2 \,\operatorname{\operatorname{\text{\sl \!Im}}} z \wedge z'\,,\quad z, z' \in {\mathbb H}\,.
\end{equation}

Returning to the tree,
the \emph{boundary} $\partial T$  of ${\mathbb T}$ is its space of \emph{ends}. Each geodesic ray in ${\mathbb T}$ 
gives rise to an end at infinity, and two rays have the same end if they coincide except 
for initial pieces of finite length. In our view on ${\mathbb T}$ as in \cite[Figure 3]{BSSW2}, 
the tree has one end $\varpi$ at the bottom, and all other ends at the top of the picture,
forming the boundary part $\partial^* {\mathbb T}$. 
For any $w \in {\mathbb T}$ and $\xi \in \partial {\mathbb T}$, there is a unique geodesic 
ray $\geo{w\,\xi}$ starting at $w$ and having $\xi$ as its end. Analogously, for any two distinct 
$\xi, \eta \in \partial {\mathbb T}$, there is a unique bi-infinite geodesic $\geo{\eta\, \xi}$ such that
if we split it at any point, one part is a ray going to $\xi$ and the other a ray going to $\eta$.
If both of them belong to $\partial^* {\mathbb T}$ then we can also define $\xi \curlywedge \eta$ as the element $v$
on the geodesic where $\mathfrak{h}(v)$ is minimal; it is a vertex.
If $\xi \in \partial^* {\mathbb T}$ then the ``vertical'' geodesic $\geo{\varpi\,\xi}$ is the side view of 
a copy of the hyperbolic plane sitting in ${\mathsf{HT}}$, namely
${\mathbb H}_{\xi} = \{(z,w) \in {\mathsf{HT}} : w \in \geo{\varpi\,\xi}\}$.
  
The metric $\mathsf{d}_{{\mathsf{HT}}}$ of ${\mathsf{HT}}$ is induced by the 
hyperbolic arc length inside each strip: let $\mathfrak{z}_1=(z_1,w_1) , \mathfrak{z}_2=(z_2,w_2) \in {\mathsf{HT}}$. 
Let $v=w_1 \curlywedge w_2\,$. Then 
\begin{equation}\label{eq:metric}
\mathsf{d}_{{\mathsf{HT}}}\bigl(\mathfrak{z}_1,\mathfrak{z}_2) = 
\begin{cases} \mathsf{d}_{{\mathbb H}}(z_1,z_2)\,,
           \hspace*{2.5cm} \text{if}\; v \in \{ w_1, w_2 \}\,\\
  \min \{ \mathsf{d}_{{\mathbb H}}(z_1,z)+\mathsf{d}_{{\mathbb H}}(z,z_2) : z \in L_{\mathfrak{h}(v)} \}\,,& \\ \hspace*{4.2cm} \text{if}\; 
        v \notin \{ w_1, w_2 \}\,.
\end{cases}
\end{equation}
In the first case, $\mathfrak{z}_1$ and $\mathfrak{z}_2$ belong to a common copy 
${\mathbb H}_{\xi}$ of ${\mathbb H}$ in ${\mathsf{HT}}$. In the second case, $v \in V({\mathbb T})$, and there are 
$\xi_1, \xi_2 \in \partial^*{\mathbb T}$  such that $\xi_1 \curlywedge \xi_2 =v$ and both 
$\mathfrak{z}_i$ lie above $\mathsf{L}_v$ in ${\mathbb H}_{\xi_i}\,$, 
whence any geodesic from $\mathfrak{z}_1$ to $\mathfrak{z}_2$ must pass through a (unique!) point $\mathfrak{z} \in \mathsf{L}_v$. 
See \cite[Figure~5]{BSSW2}.

\smallskip

Next, we recall the isometry group of ${\mathsf{HT}}(\mathsf{q},\mathsf{p})$.
First, the locally compact group of affine transformations
\begin{equation}\label{eq:AffRq}
\operatorname{\sf Aff}({\mathbb H},\mathsf{q}) = \left\{ g=\begin{pmatrix} \mathsf{q}^n & b \\ 0 & 1 \end{pmatrix} :
n \in {\mathbb Z}\,,\;b\in \mathbb{R} \right\}\,,\quad gz = \mathsf{q}^n z + b 
\end{equation}
acts on ${\mathbb H}(\mathsf{q})$ by isometries and preserves the set of ``slicing'' lines $\mathsf{L}_k$
($k \in {\mathbb Z}$). Left Haar measure $dg$ and its modular function 
${\boldsymbol\delta}_{{\mathbb H}} = {\boldsymbol\delta}_{{\mathbb H},\mathsf{q}}$ are 
$$
dg = \mathsf{q}^{-n}\,dn\,db \quad\text{and}\quad {\boldsymbol\delta}_{{\mathbb H}}(g) = \mathsf{q}^{-n}\,,\quad\text{if}\quad
g={\scriptsize \begin{pmatrix}\mathsf{q}^n& b \\ 0 & 1 \end{pmatrix}}\,,
$$
where $dn$ is counting measure on ${\mathbb Z}$ and $db$ is Lebesgue measure on $\mathbb{R}$. 

Second, denote by $\operatorname{\sf Aut}({\mathbb T}_{\mathsf{p}})$ the full isometry group of ${\mathbb T}_{\mathsf{p}}$ and consider
the \emph{affine group} of ${\mathbb T}_{\mathsf{p}}\,$,
\begin{equation}\label{eq:AffTp}
\operatorname{\sf Aff}({\mathbb T}_{\mathsf{p}}) = \{ \gamma \in \operatorname{\sf Aut}({\mathbb T}_{\mathsf{p}}) : (\gamma v)^-= v^- \;\text{for all}\; v \in V({\mathbb T}_{\mathsf{p}}) \}\,. 
\end{equation}
The modular function 
${\boldsymbol\delta}_{{\mathbb T}} = {\boldsymbol\delta}_{{\mathbb T}_{\mathsf{p}}}$ of $\operatorname{\sf Aff}({\mathbb T}_{\mathsf{p}})$ is given by
$$
{\boldsymbol\delta}_{{\mathbb T}}(\gamma) = \mathsf{p}^{\Phi(\gamma)} \quad \text{where}\quad 
\Phi(\gamma)=\mathfrak{h}(\gamma w) - \mathfrak{h}(w) \,,\quad\text{if} 
\quad \gamma \in \operatorname{\sf Aff}({\mathbb T}_{\mathsf{p}})\,,\;w \in {\mathbb T}\,.
$$
The above mapping $\Phi: \operatorname{\sf Aff}({\mathbb T}) \to {\mathbb Z}$  
is independent of $w \in {\mathbb T}$ and a homomorphism.
\begin{thm}\label{thm:isogroup}\cite{BSSW2} The group
$$
\mathcal{A} = \mathcal{A}(\mathsf{q},\mathsf{p}) = \{ (g,\gamma) \in \operatorname{\sf Aff}({\mathbb H},\mathsf{q}) \times \operatorname{\sf Aff}({\mathbb T}_{\mathsf{p}}) : 
\log_{\mathsf{q}} {\boldsymbol\delta}_{{\mathbb H}}(g) + \log_{\mathsf{p}} {\boldsymbol\delta}_{{\mathbb T}}(\gamma) = 0 \} 
$$
acts  by isometries $(g,\gamma)(z,w) = (gz,\gamma w)$ on ${\mathsf{HT}}(\mathsf{q},\mathsf{p})$, 
with compact quotient isomorphic with the circle of length
$\log \mathsf{q}$. It leaves the area element $\,d\mathfrak{z}$ of $\,{\mathsf{HT}}$ as well as the transition kernel 
\eqref{eq:pab} invariant:
\begin{equation}\label{eq:groupinvariance}
\mathbf{p}_{\alpha,\beta}(t,\mathfrak{g}\mathfrak{w},\mathfrak{g}\mathfrak{z}) = \mathbf{p}_{\alpha,\beta}(t,\mathfrak{w},\mathfrak{z}) 
\end{equation}
for all $\,t > 0\,,\; \mathfrak{w}, \mathfrak{z} \in {\mathsf{HT}} \;$ and $\; 
\mathfrak{g} \in \mathcal{A}\,$.
\end{thm}

Next, we recall from \cite{BSSW2} the stopping times $\tau(n)\;$ ($n \in {\mathbb N}_0$)
of the successive visits of $(Y_t)$ in ${\mathbb Z}$
\begin{equation}\label{eq:stop}
\tau(0)=0, \quad \tau(n+1) 
= \inf\bigl\{ t > \tau(n) : Y_t \in {\mathbb Z} \setminus \{ Y_{\tau(n)} \}\bigr\}\,.
\end{equation}
Via Proposition \ref{pro:projections}, we can also interpret them in terms
of Brownian motion on ${\mathsf{HT}}$. If $X_0$ lies in $\mathsf{S}_v^o$, then $\tau(1)$ is
the instant when $X_t$ first meets a point on $\mathsf{L}_v \cup \mathsf{L}_{v^-}\,$.  If 
$X_{\tau(n)} \in \mathsf{L}_v$ for some $v \in V({\mathbb T})$
(which holds for all $n \ge 1$, and possibly also for $n=0$), then
$\tau(n+1)$ is the first instant $t > \tau(n)$ when $X_t$ meets one of the
bifurcation lines $\mathsf{L}_{v^-}$ or $\mathsf{L}_w$ with $w^- = v$. 
The increments $\tau(n)-\tau(n-1)$, $n \ge 1$, are independent and almost surely finite.
They are identically distributed for $n\ge 2$, and when $Y_0 \in {\mathbb Z}$ (equivalently, $Z_0 \in \operatorname{\sf LT}$),
then also $\tau(1)$ has the same distribution.
The random variables
\begin{equation}\label{eq:tau}
\tau = \tau(2) - \tau(1) \quad\text{and}\quad Y = Y_{\tau(2)} - Y_{\tau(1)}
\end{equation}
are independent, 
\begin{equation}\label{eq:rha}
\mathsf{Pr}[Y = 1] = \frac{\mathsf{a}}{\mathsf{a}+1} \quad\text{and}\quad  
\mathsf{Pr}[Y = -1] = \frac{1}{\mathsf{a}+1}\,,\quad 
\mathsf{a} = \beta \mathsf{p} \,\mathsf{q}^{\alpha-1}\,,
\end{equation}
and $\tau$ has finite exponential moment $\mathsf{E}(e^{\lambda_0 \tau})$ for some $\lambda_0 > 0$.

\smallskip

For any open domain $\Theta \subset {\mathsf{HT}}$, we let 
\begin{equation}\label{eq:exit}
 \tau^{\Theta} = \inf \{ t>0: X_t \in {\mathsf{HT}} \setminus \Theta\}
\end{equation}
be the \emph{first exit time} of $(X_t)$ from $\Theta$.

If $\tau = \tau^{\Theta} < \infty$ almost surely for the starting point
$X_0 = \mathfrak{w} \in \Theta$, then we write $\mu_{\mathfrak{w}}^{\Theta}$ for the
distribution of $X_{\tau}\,$. In all cases that we shall consider, 
$\mu_{\mathfrak{w}}^{\Theta}$ is going to be a probability measure supported by $\partial \Theta$.
We shall use analogous notation on ${\mathbb H}$, ${\mathbb T}$ and $\mathbb{R}$. We note that
by group invariance of our Laplacian,
\begin{equation}\label{eq:inv-meas}
\mu_{\mathfrak{w}}^{\Theta}(B) = \mu_{\mathfrak{g}\mathfrak{w}}^{\mathfrak{g}\Theta}(\mathfrak{g} B) \quad\text{for every}\; 
\mathfrak{g} \in \mathcal{A} \;\text{and Borel set}\; B \subset {\mathsf{HT}} \,.
\end{equation}

We conclude this section with the ``official'' definition of harmonic functions.

\begin{dfn}\label{def:harmonic} Let $\Theta \subset {\mathsf{HT}}$ be open.
A continuous function $f:\Theta  \to \mathbb{R}$ is called \emph{harmonic on $\Theta$}
if for every open, relatively compact set $U$ with $\overline U \subset \Theta$,
$$
f(\mathfrak{z}) = \int f\, d\mu_{\mathfrak{z}}^U \quad \text{for all}\; \mathfrak{z} \in U\,.
$$
\end{dfn}

As mentioned in \cite{BSSW2}, from a classical analytic viewpoint, 
``harmonic'' should rather mean ``annihilated by the Laplacian'', but
for general open domains in ${\mathsf{HT}}$, the correct formulation in these terms 
is subtle in view of the relative location of the bifurcations.  
By \cite[Theorem 5.9]{BSSW1}, if $\Theta$ is suitably ``nice'', then any
harmonic function on $\Theta$ is anihilated by $\Delta_{\alpha,\beta}\,$. This is true in 
particular for the sets $\Omega_r$ that we are going to use in the next section; see 
Figure~2 below. More details will be stated and used later on; in particular, see Proposition
\ref{pro:harmonic} regarding globally harmonic functions on ${\mathsf{HT}}$.

\section{Harmonic functions on rectangular sets}\label{sec:rectangular}
 
In this section, we want to derive a Poisson representation formula
for harmonic functions on open domains $\Omega \subset {\mathsf{HT}}$ which 
have a rectangular shape. Prototypes are the following sets, where 
$v \in V({\mathbb T})$ and $r > 0$.
\begin{equation}\label{eq:Omv}
\begin{gathered}
\Omega_v = \{ (z,w) \in {\mathsf{HT}} : w \in N(v)^o\} \subset {\mathsf{HT}}\,, \quad\text{and}\quad \\ 
\Omega_{v,r} = \{ \mathfrak{z} \in \Omega_v : |\operatorname{\text{\sl \!Re}} \mathfrak{z}| < r \}\,,
\end{gathered}
\end{equation}
where $N(v)$ is the ``neighbourhood star'' at $v$ in ${\mathbb T}$. 
That is, $N(v)$ is the union of all edges ($\equiv$ intervals !) of ${\mathbb T}$ which
have $v$ as one endpoint. It is a compact metric subtree of ${\mathbb T}$, whose boundary
$\partial N(v)$ consists of all neighbours of $v$ in $V({\mathbb T})$.  
 We write
$\partial^+ N(v) = \partial N(v) \setminus \{ v^-\}$ (the forward neighbours
of $v$).

Note that we can define the space
${\mathcal C}^\infty(\Omega)$ for $\Omega=\Omega_v$ or $\Omega_{v,r}$ in the same way as in 
Definition \ref{def:Cinfty} by restricting to the (pieces of the) strips that make up $\Omega\,$.
We recall the following \cite[Theorem 5.9]{BSSW1}:

\begin{pro}\label{pro:harmonic}
A function $f$ is harmonic on $\Omega = {\mathsf{HT}}$, resp. $\Omega=\Omega_v\,$, resp.
$\Omega=\Omega_{v,r}\,$ if and only if 
$f \in \mathcal C^{\infty}(\Omega)$, it satisfies the bifurcation conditions
\eqref{eq:bif} at each bifurcation line inside $\Omega$, 
and $y^{2}(\partial_x^2+\partial_y^2)f
+\alpha \,y\, \partial_yf = 0$ in each open strip in $\Omega$.
\end{pro}

For deriving a Poisson representation for harmonic
functions on $\Omega_v\,$, it is sufficient to 
consider just $\Omega_o\,$ (by group invariance). We shall henceforth 
always reserve the letter 
$\Omega$ for this specific set. We follow classical reasoning,
but here we have to be careful in view of the singularities of our domain
along the bifurcation line $\mathsf{L}_o\,$; the validity of each step has to
be checked. 

We first work with the bounded subsets $\Omega_r = \Omega_{o,r}\,$.
Here, $r > 0$ will always be a real number, so that $\Omega_r$ should not be
confounded with $\Omega_v$ for $v \in V({\mathbb T})$. 
The boundary $\partial \Omega_r$ in ${\mathsf{HT}}$ consists of the 
\emph{horizontal part} and the \emph{vertical part}, given by 
\begin{equation}\label{eq:hr-vr}
\begin{aligned}
\partial^{\text{\rm hor}} \Omega_r 
&= \{ (z,v) \in {\mathsf{HT}} : v \in \partial N(o)\,,\; |\operatorname{\text{\sl \!Re}} z| \le r\} \quad\text{and}\quad\\
\partial^{\text{\rm vert}} \Omega_r &= N_r(o) \cup N_{-r}(o)\,,\quad \text{where for any}\;
r \in \mathbb{R}\,,\\
N_r(o) &= \{ (z,w) \in {\mathsf{HT}} : w \in N(o)\,,\; \operatorname{\text{\sl \!Re}} z = r\}\,,
\end{aligned}
\end{equation}
respectively. 
The horizontal part lies on the union of all bifurcation 
lines $\mathsf{L}_v\,$, $v \in \partial N(o)$. The vertical part consists of two 
isometric copies of the compact metric tree $N(o)$ within ${\mathsf{HT}}$, which delimit 
our set on the left 
and right hand sides, at which $\operatorname{\text{\sl \!Re}} z = - r$, resp. $\operatorname{\text{\sl \!Re}} z = r$.
We call the elements of the finite set $\bigl\{ (z,v) \in {\mathsf{HT}}: |\operatorname{\text{\sl \!Re}} z| = r\,,\ 
v \in V\bigl(N(o)\bigr) \bigr\}$ the \emph{corners} of~$\Omega_r\,$. 
$$
\beginpicture 

\setcoordinatesystem units <1mm,1mm> 

\setplotarea x from -16 to 54, y from -20 to 36

\plot 0 0  40 20  48 36  8 16  0 0  -8 16  32 36  35.2 29.6 / 

\plot 0 0  9 -18  49 2  40 20 /

\put {$\bullet$} at 20 10

\put {$\mathfrak{o}$} at 21.5 8.5
\multiput{$\circ$} at   0 0  9 -18  49 2  40 20
-8 16   8 16  32 36  48 36 /
\endpicture
$$
\begin{center}
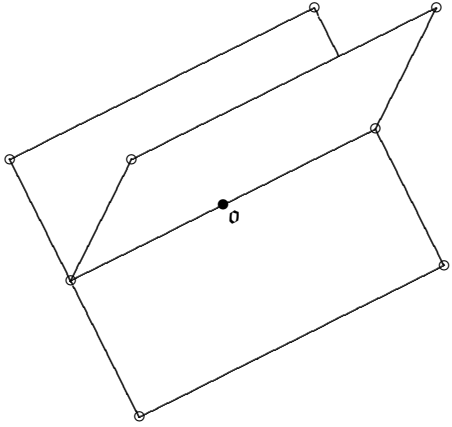
{\sl Figure 2.\/} The compact set $\overline\Omega_r\,$. 
Its corners are encircled.
\end{center}

\smallskip

Our $\overline\Omega_r$ is the union of the finitely many closed
``rectangles'' $\overline R_{v,r}\,$, where
$$
R_{v,r} = \{ \mathfrak{z} \in \mathsf{S}_v^o : |\operatorname{\text{\sl \!Re}} \mathfrak{z}| < r \}.
$$
The boundary of $R_{v,r}$ consists of two horizontal sides
$L_{v,r}$ and $L_{v^-,r}$ and two vertical 
sides $J_{v,r}$ and $J_{v,-r}\,$, where $L_{v,r}= \{ (z,v) \in \mathsf{L}_v: |z| \le r \}$ and 
$J_{v,r} = \{ (z,w) \in {\mathsf{HT}}: w \in [v^-,v]\,,\; \operatorname{\text{\sl \!Re}} z = r \}$.
Here, the vertex $v$ ranges in $\{ o\} \cup \partial^+N(o)$, and
the closed rectangles $\overline R_{v,r}$ and $\overline R_{o,r}$ meet at their lower, resp.
upper horizontal sides, when $v^- = o$.

We want to write down Green's formulas for ``nice'' functions on 
$\overline\Omega_r\,$. We start by specifying those functions.

\begin{dfn}\label{def:D2}
Let $\mathcal D^2_{\alpha,\beta}(\overline\Omega_r)$ be the space of all 
continuous functions $f$ on $\overline\Omega_r$ with the following properties.
\begin{itemize}
\item[(i)]
For $v=o$ as well as for any forward neighbour $v$ of $o$ in $V({\mathbb T})$, 
the partial derivatives of $f$ up to $2$nd order exist and are bounded 
and continuous on $R_{v,r}\,$, and 
\item[(ii)] they extend continuously from $R_{v,r}$
to $\partial R_{v,r}\,$, except possibly at the corners of
$\Omega_r$. (But on bifurcation lines, the extensions coming from 
different strips will in general not coincide.) 
\item[(iii)] The first order derivatives satisfy the bifurcation condition
\eqref{eq:bif} along the segment $L_{o,r}\,$.
\item[(iv)] $\Delta_{\alpha,\beta} f$ extends continuously to $\Omega_r\,$ 
across the bifurcation line.
\end{itemize}
If $U \subset \Omega_r$ is a ``nicely shaped'' open subset of
$\Omega_r$, we define the space 
$\mathcal D^2_{\alpha,\beta}(\overline\Omega_r \setminus U)$ analogously.
Furthermore,  we define 
$\mathcal D^2_{\alpha,\beta}(\overline\Omega)$ as the space of all functions on
$\overline\Omega$ whose restriction to $\overline\Omega_r$ is in 
$\mathcal D^2_{\alpha,\beta}(\overline\Omega_r)$ for every $r > 0$.
\end{dfn}

Next, we need to specify the 
boundary (arc) measure induced by the measure $\mathbf{m} = \mathbf{m}_{\alpha,\beta}$ of
\eqref{eq:mab}. (We shall omit the paramters $\alpha, \beta$ in the index.)
More precisely, we need one such arc measure $\mathbf{m}_v' = \mathbf{m}_{\alpha,\beta;v}'$
on each closed strip $\mathsf{S}_v\,$, where $v \in V({\mathbb T})$,
given by
\begin{equation}\label{eq:mab'}
d\mathbf{m}_v'(\mathfrak{z})= \phi_{\alpha,\beta}(\mathfrak{z})\,\sqrt{dx^2 + dy^2}\big/y 
\quad\text{for}\quad \mathfrak{z}=(x+\mathfrak{i}\, y\,; v)\in \mathsf{S}_v\,,
\end{equation}
where $\phi_{\alpha,\beta}(\mathfrak{z})$ is as in \eqref{eq:mab}. 
Contrary to the definition \eqref{eq:mab} of $\mathbf{m}$, the lower 
boundary line $\mathsf{L}_{v^-}$ of $\mathsf{S}_v$ is not excluded in this definition, so that
on that line we have the two arc measures $\mathbf{m}_v'$ and $\mathbf{m}_{v^-}'$ 
which differ by a factor of $\beta$. Their use depends on the side from which
the line is approached.
Now we can construct the boundary measure 
$\mathbf{m}^{\partial \Omega_r}$ and the boundary gradient of 
$f \in \mathcal D^2_{\alpha,\beta}(\overline \Omega_r)$ as follows:
\\[3pt]
(1) for each $v \in \partial^+N(o)$, we let
$\mathbf{m}^{\partial \Omega_r} = \mathbf{m}_v'$ and 
$\nabla^{\partial \Omega_r}f = \nabla f_v$ (the extension to $\partial \Omega_r$
within $\overline{\mathsf{S}}_v$) on the boundary parts $L_{v,r}$ and on $J_{v,\pm r}\,$;
\\[3pt]
(2) furthermore, $\mathbf{m}^{\partial \Omega_r} = \mathbf{m}_o'$ and 
$\nabla^{\partial \Omega_r}f = \nabla f_o$ (the extension to $\partial \Omega_r$
within $\overline{\mathsf{S}}_o$) on the boundary parts $L_{o^-,r}$ and on $J_{o,\pm r}\,$.
\\[3pt]
In (1) and (2), the corners of $\Omega$ -- a set with measure $0$ -- are
excluded. In the following, the occuring boundaries are considered with
positive orientation.  

\begin{pro}\label{pro:green} For $f, h \in \mathcal D^2_{\alpha,\beta}(\overline\Omega_r)$,
$$
\begin{aligned}
\int_{\Omega_r} \!\!\Bigl( f \,\Delta_{\alpha,\beta}h + (\nabla f, \nabla h)\Bigr)
\, d\mathbf{m} &=\! \int_{\partial \Omega_r}\!\! (\mathsf{n}, \nabla^{\partial \Omega_r}h)\, f 
\,  d\mathbf{m}^{\partial \Omega_r} \quad\text{and}\quad\\
\int_{\Omega_r}\!\! \Bigl( f \,\Delta_{\alpha,\beta}h - h\,\Delta_{\alpha,\beta}f)\Bigr)
\, d\mathbf{m} &=\! \int_{\partial \Omega_r}\!\! \Bigl((\mathsf{n}, \nabla^{\partial \Omega_r}h)\, f - 
(\mathsf{n}, \nabla^{\partial \Omega_r}f)\, h\Bigr)\,d\mathbf{m}^{\partial \Omega_r},
\end{aligned}
$$
where $\mathsf{n}(\mathfrak{z}) = \mathsf{n}^{\partial \Omega_r}(\mathfrak{z})$ is the outward unit normal vector
to $\partial \Omega_r$ at $\mathfrak{z} \in \partial \Omega_r\,$,
defined except at the corners of $\Omega_r\,$.
\end{pro}

\begin{proof} 
We can write down the classical first Green formula as above on each
$R = R_{v,r}$, with the boundary gradient $\nabla^{\partial R}f$  :
$$
\int_{R_{v,r}} \Bigl( f \,\Delta_{\alpha,\beta}h + (\nabla f, \nabla h)\Bigr)
\, d\mathbf{m} = \int_{\partial R_{v,r}} (\mathsf{n}, \nabla^{\partial R}h)\, f 
\,  d\mathbf{m}'_v\,.
$$
Here, as above for $\Omega_r$ above, the boundary gradient 
$\nabla^{\partial R}f$ is the continuous extension of the gradient 
to $\partial R_{v,r}\,$, well defined except at the corners. 

Indeed, we can first inscribe a slightly smaller rectangle whose closure is
contained in $R_{v,r}\,$. We then have Green's first formula on that set, as 
stated in
any calculus textbook. Then we can increase that inscribed rectangle and
pass to the limit. The exchange of limit and integrals is legitimate because
of the boundedness assumptions that appeared in the definition of the space
$\mathcal D^2_{\alpha,\beta}(\overline\Omega)$.

\smallskip

Now we take the sum of those integrals over all $R_{v,r}$ that make up
$\Omega_r\,$. The segment $L_{o,r}$ of the bifurcation
line $\mathsf{L}_o$ lies in the interior of $\Omega$. The bifurcation condition 
makes sure that the line integrals over $L_{o,r}$ coming from the 
different adjacent rectangles sum up to $0$. 

Green's second formula follows from the first one.
\end{proof}

By \cite[Theorems 5.9 and 5.19]{BSSW1}, the set $\Omega_r$ admits a Dirichlet 
Green kernel $\mathbf{g}^{\Omega_r}(\cdot, \mathfrak{w})$ which for $\mathfrak{w} \in \Omega_r$ belongs to 
$\mathcal D^2_{\alpha,\beta}(\overline\Omega_r\setminus U)$ for any open neighbourhood
$U$ of its pole $\mathfrak{w}$.
The fact that $\mathbf{g}^{\Omega_r}(\cdot, \mathfrak{w})$ is smooth up to the boundary segments
of our $\Omega_r$ (with exception of the corners, where, however, the partial
derivatives are bounded) follows from the general theory of ellptic 2nd order
PDEs; see {\sc Evans}~\cite[Theorem 6, page 326]{Eva}. 

We use it to define the \emph{Poisson kernel} 
\begin{equation}\label{eq:Poiss}
\Pi^{\Omega_r}(\mathfrak{w},\mathfrak{z})= 
-\Bigl(\mathsf{n}^{\partial \Omega_r}(\mathfrak{z})\,,\, 
\nabla_{\mathfrak{z}}^{\partial \Omega_r}\mathbf{g}^{\Omega_r}(\mathfrak{z}, \mathfrak{w})\Bigr)\,,\quad 
\mathfrak{w} \in \Omega_r\,,\; 
\mathfrak{z} \in \partial \Omega_r\,.
\end{equation}
(The gradient is applied to the variable $\mathfrak{z}$ in the index, or rather to
its hyperbolic part $z = x+\mathfrak{i}\, y$).)
Note that the Poisson kernel is positive.

The argument used for the next lemma is classical, but it is 
important to check that each step works in our ``non-classical'' setting 
with the bifurcation lines and Kirchhoff condition.

\begin{lem}\label{lem:poiss}{\bf [Poisson representation]}
Let $h \in \mathcal D^2_{\alpha,\beta}(\overline\Omega_r)$
be harmonic, that is, $\Delta_{\alpha,\beta}h = 0$ in $\Omega_r\,$.
Then for every $\mathfrak{w} \in \Omega_r\,$,
$$
h(\mathfrak{w}) = \int_{\partial \Omega_r} \Pi^{\Omega_r}(\mathfrak{w},\mathfrak{z}) h(\mathfrak{z}) \, 
d\mathbf{m}^{\partial \Omega_r}(\mathfrak{z})\,.
$$
\end{lem}

\begin{proof}
Let $\varphi$ be any  non-negative function in ${\mathcal C}^\infty_c(\Omega_r)$
and let $f$  
be its Green potential on $\Omega_r\,$, that is,
\begin{equation}\label{eq:Greenpot}
f(\mathfrak{w}) = \int_{\Omega_r} \mathbf{g}^{\Omega_r}(\mathfrak{w},\mathfrak{z})\,\varphi(\mathfrak{z})\, d\mathbf{m}(\mathfrak{z})\,.
\end{equation}
Then $f$ is a weak solution of the equation 
$$
\Delta_{\alpha,\beta}f = -\varphi \quad \text{in}\; \Omega_r\,.
$$

Smoothness of the heat kernel $\mathbf{h}_{\alpha,\beta}(t,\mathfrak{w},\mathfrak{z})$ associated with 
$\Delta_{\alpha,\beta}$ (see \cite[Theorem 5.23 and Appendix 1]{BSSW1})
implies as in the classical theory of PDE that $f$ is a strong solution
of the abvove Laplace equation, that is, $f \in D^2_{\alpha,\beta}(\overline\Omega_r)$
(indeed, it has all higher order derivatives).
We have $f = 0$ on $\partial \Omega_r$.

Inserting $f$ and $h$ into Green's second identity of Proposition 
\ref{pro:green},
we get with an application of Fubini's theorem
$$
\begin{aligned} \int_{\Omega_r} h(\mathfrak{w})\,&\varphi(\mathfrak{w}) \, d\mathbf{m} 
= - \int_{\partial \Omega_r} \Bigl(\mathsf{n}(\mathfrak{z}), \nabla_{\mathfrak{z}}^{\partial \Omega_r}f(\mathfrak{z})
\Bigr)\, h(\mathfrak{z})  \,d\mathbf{m}^{\partial \Omega_r}(\mathfrak{z})\\
&= \int_{\partial \Omega_r} \int_{\Omega_r} 
\Bigl(\mathsf{n}(\mathfrak{z}),\nabla_{\mathfrak{z}}^{\partial \Omega_r} \mathbf{g}^{\Omega_r}(\mathfrak{w},\mathfrak{z}) 
\,\varphi(\mathfrak{z})\Bigr)
\,h(\mathfrak{z})\,d\mathbf{m}(\mathfrak{w})  \, d\mathbf{m}^{\partial \Omega_r}(\mathfrak{z})\\
&= \int_{\Omega_r} \left(
\int_{\partial \Omega_r} \Pi^{\Omega_r}(\mathfrak{w},\mathfrak{z}) h(\mathfrak{z})
\,d\mathbf{m}^{\partial \Omega_r}(\mathfrak{z}) \right) \varphi(\mathfrak{w})
\, d\mathbf{m}(\mathfrak{w})\,.
\end{aligned}
$$    
Since this holds for any continuous function $\varphi$ as specified above, 
the statement follows.
\end{proof}

Let $f$ be a continuous function on $\partial \Omega_r\,$.
Then
$$
h(\mathfrak{w}) = \int_{\partial \Omega_r} f\, d\mu_{\mathfrak{w}}^{\Omega_r} =
\mathsf{E}_{\mathfrak{w}}\Bigl(f(X_{\tau^{\Omega_r}})\Bigr)
$$
is the unique solution of the Dirichlet problem with boundary data $f$.
By Lemma \ref{lem:poiss},
$$
h(\mathfrak{w}) = \int_{\partial \Omega_r} \Pi^{\Omega_r}(\mathfrak{w},\mathfrak{z}) f(\mathfrak{z}) \, 
d\mathbf{m}^{\partial \Omega_r}(\mathfrak{z})\,.
$$
\begin{cor}\label{cor:dens} For any $\mathfrak{w} \in \Omega_r\,$, the Poisson kernel 
$\Pi^{\Omega_r}(\mathfrak{w},\cdot)$ is the density of the hitting ($\equiv$ exit) distribution
$\mu_{\mathfrak{w}}^{\Omega_r}$ with respect to the boundary measure 
$\mathbf{m}^{\partial \Omega_r}$.
\end{cor}

We want to have an analogous 
result for the non-compact set $\Omega = \Omega_o\,$. 
In spite of being ``obvious'', this requires 
further substantial work. 

\begin{lem}\label{lem:measure1} For any $\mathfrak{w} \in \Omega$ and
any $r > |\operatorname{\text{\sl \!Re}} \mathfrak{w}|$,
$$
\mu_{\mathfrak{w}}^{\Omega_r}  \bigl( \partial^{\text{\rm vert}}\Omega_r \bigr) \le 
\mathsf{Pr}_{\mathfrak{w}} \bigl[ \max \{ |\operatorname{\text{\sl \!Re}} X_t| : t \le \tau^{\Omega}\} > r\bigr] \le
2\rho^{r - |\operatorname{\text{\sl \!Re}} \mathfrak{w}|-1}  \,,
$$
where $\rho <1$. 
\end{lem}

\begin{proof} The first inequality is clear: if up to the exit time from the entire
set $\Omega$ we have $|\operatorname{\text{\sl \!Re}} X_t| \le r$
then $X_t$ cannot exit $\Omega_r$ through its vertical boundary part.
For the second inequality, the arguments are those
used in the proof of \cite[Proposition 4.18]{BSSW2}. In particular, see
\cite[Figure 7]{BSSW2}. The number $\rho$ is as in that proposition. The result there
is formulated in terms of the projected process on ${\mathbb H}(\mathsf{q})$, and we add some
brief reminders in terms of ${\mathsf{HT}}$ itself.

We ``slice'' $\Omega$ vertically by the sets $N_{\pm k}(o)$, $k \in {\mathbb N}$, and let $\sigma(k)$ be 
the exit time of our Brownian motion from $\Omega_k$, while $\tau^{\Omega}$ coincides with 
the $\sigma$ of \cite[(4.3)]{BSSW2}: with respect to
the projection $Y_t = \pi^{\mathbb{R}}(X_t)$ on the real line,
\begin{equation}\label{eq:sig}
\begin{gathered}
\tau^{\Omega} = \inf\{ t > 0 : Y_t \notin [-1, 1]\}
= \inf\{ t \ge 0 : Y_t = \pm 1\}\,, \\ 
\text{where} \quad Y_0 \in [-1\,,\,1],
\end{gathered}
\end{equation}
which is a.s. finite and has an exponential moment. We also note that $\sigma(k) \le \tau^{\Omega}$.

The function 
$\mathfrak{z} \mapsto \mathsf{Pr}_{\mathfrak{z}}[X_{\sigma(1)} \in \partial^{\text{\rm hor}} \Omega_1]$ is strictly 
positive and weakly harmonic
($\equiv$ in the sense of distributions) on $\Omega_1\,$, whence
strongly harmonic by \cite[Theorem 5.9]{BSSW1}, and consequently continuous.  
Along $N_0(o)$, it must attain its minimum  in some 
$\mathfrak{z}_* \in N_0(o) \setminus \partial \Omega_1\,$. Then
$$
\rho = 1 - \mathsf{Pr}_{\mathfrak{z}_*}[X_{\sigma(1)} \in \partial^{\text{\rm hor}} \Omega_1].
$$
Now let us assume first that $\mathfrak{w} \in \Omega$ is such that $\operatorname{\text{\sl \!Re}} \mathfrak{w} = 0$. 
If $X_t$ starts at $\mathfrak{w}$ and $\max \{ \operatorname{\text{\sl \!Re}} X_t : t \le \tau^{\Omega} \} > r$, 
then it must
pass through all the vertical ``barriers'' $N_1(o)\,,\cdots, N_n(o)$, where $n = \lfloor r \rfloor$. 
But by group-invariance (just using horizontal translations here), 
for any starting point in $N_{k-1}(o)$, the probability that $X_t$ reaches 
$N_k(o) \setminus \partial \Omega$ before $\partial \Omega$ is bounded above
by $\rho$. By a simple inductive argument as in 
in the proof of \cite[Proposition 4.18]{BSSW2}, the probability to pass through all those
$N_k(o)$ before exiting $\Omega$ is $\le \rho^n \le \rho^{r - 1}$. 

In the same way, $\mathsf{Pr}_{\mathfrak{w}} \bigl[\min \{ \operatorname{\text{\sl \!Re}} X_t : t \le \tau^{\Omega} \} 
< -r\bigr] \le \rho^{r - 1} $.
Thus, the stated upper bound of the second inequality holds when $\operatorname{\text{\sl \!Re}} \mathfrak{w} = 0$. 

If $\mathfrak{w} \in \Omega$ is arbitrary, then we can use group invariance and map $\mathfrak{w}$ to a point $(0,w_0)$ 
by a horizontal translation of ${\mathsf{HT}}$. Then we must replace $\Omega_r$ by $\Omega_{r-|\operatorname{\text{\sl \!Re}} \mathfrak{w}|}\,$ in
the preceding arguments.
\end{proof}

Since $\tau^{\Omega} = \sigma$ is a.s. finite, also the set $\Omega$ has a Dirichlet
Green kernel $\mathbf{g}^{\Omega}(\cdot, \mathfrak{w})$ which for $\mathfrak{w} \in \Omega$ belongs to 
$\mathcal D^2_{\alpha,\beta}(\overline\Omega \setminus U)$ for every ``nicely shaped''
open neighbourhood $U$ of its pole $\mathfrak{w}$. Here,
the space $\mathcal D^2_{\alpha,\beta}(\overline\Omega)$ is as
in Definition \eqref{def:D2}, except that we do not have to refer to the corners,
and boundedness of partial derivatives is local, i.e., it
refers to arbitrary relatively compact subsets of any of the strips that make
up $\Omega$.

The following is another relative of \cite[Proposition 4.18]{BSSW2} and the above  Lemma
\ref{lem:measure1}.

\begin{lem}\label{lem:zero} For any $\mathfrak{w} \in \Omega$,
$$
\mathbf{g}^{\Omega}(\mathfrak{z}, \mathfrak{w}) \to 0 \quad\text{as }\; d(\mathfrak{z},\mathfrak{w}) \to \infty\,,\;
\mathfrak{z} \in \Omega\,.
$$
\end{lem}

\begin{proof}
We use classical \emph{balayage}, see {\sc Dynkin}~\cite[(13.82) in \S 13.25]{Dy}
and {\sc Blumenthal and Getoor}~\cite[Theorem (1.16) on p. 261]{BlGe}.
We have for all $\mathfrak{z}, \mathfrak{w} \in \Omega$ and
$r > |\operatorname{\text{\sl \!Re}} \mathfrak{w}|$
\begin{equation}\label{eq:bala}
\mathbf{g}^{\Omega}(\mathfrak{w}, \mathfrak{z}) - \mathbf{g}^{\Omega_r}(\mathfrak{w}, \mathfrak{z}) = 
\int_{\partial \Omega_r}  \mathbf{g}^{\Omega}(\cdot, \mathfrak{z})\, 
d\mu_{\mathfrak{w}}^{\Omega_r} 
= \int_{\partial^{\text{\rm vert}} \Omega_r} \mathbf{g}^{\Omega}(\cdot, \mathfrak{z})\,
d\mu_{\mathfrak{w}}^{\Omega_r} \,,
\end{equation}
where the last equality holds because $\mathbf{g}^{\Omega}(\cdot, \mathfrak{z}) =0$ 
on $\partial^{\text{\rm hor}} \Omega_{r}\,$.
Of course, we have $\mathbf{g}^{\Omega_r}(\mathfrak{w}, \mathfrak{z}) = 0$ unless also
$\mathfrak{z} \in \Omega_r\,$.

For fixed $\mathfrak{w}$, when $d(\mathfrak{z},\mathfrak{w})$ is sufficiently
large,  there is $r = r(\mathfrak{z})$ such that $\mathfrak{w} \in \Omega_r$ and 
$\mathfrak{z} \in \Omega \setminus \Omega_r\,$, and $r(\mathfrak{z}) \to \infty$ when
$d(\mathfrak{z},\mathfrak{w}) \to \infty\,$.   
Then
$$
\mathbf{g}^{\Omega}(\mathfrak{z}, \mathfrak{w}) = \mathbf{g}^{\Omega}(\mathfrak{w}, \mathfrak{z}) \le 
\mu_{\mathfrak{w}}^{\Omega_{r}} \bigl( \partial^{\text{\rm vert}}\Omega_r\bigr) \times
\sup \{ \mathbf{g}^{\Omega}(\mathfrak{u}, \mathfrak{z}) : \mathfrak{u} \in \partial^{\text{\rm vert}} \Omega_r\}. 
$$
By Lemma \ref{lem:measure1}, this tends to $0$, as $r = r(\mathfrak{z}) \to \infty\,$.
\end{proof}

We can now define the Poisson kernel $\Pi^{\Omega}(\mathfrak{w},\mathfrak{z})$
as in \eqref{eq:Poiss}. It is positive and supported by $\partial \Omega\,$.
We have $\Pi^{\Omega}(\mathfrak{w},\mathfrak{z}) = \partial_y \mathbf{g}^{\Omega}(\mathfrak{z}, \mathfrak{w})$
on $\mathsf{L}_{o^-}$ and 
$\Pi^{\Omega}(\mathfrak{w},\mathfrak{z}) = -\partial_y \mathbf{g}^{\Omega}(\mathfrak{z}, \mathfrak{w})$
on each $\mathsf{L}_v$ with $v^- = o$. (The partial deriviative $\partial_y$ refers
to the $y$-coordinate of $\mathfrak{z} = x + \mathfrak{i}\, y$.)

\begin{pro}\label{pro:one}\hspace{1cm}
$
\int_{\partial \Omega} \Pi^{\Omega}(\mathfrak{w},\mathfrak{z})\, d\mathbf{m}^{\partial \Omega} = 1\,.
$
\end{pro}

\begin{proof}
For $r > 0$, we use Green's second identity of Proposition 
\ref{pro:green} on $\Omega_r\,$. We again let $\varphi$ be any  non-negative 
function in ${\mathcal C}^\infty_c(\Omega_r)$. We use the Green potential 
of $\varphi$ as in \eqref{eq:Greenpot}, but this time with respect to the whole 
of $\Omega$ instead  of $\Omega_r\,$. Furthermore, we let $h \equiv 1$.
We have $f = 0$ on $\partial \Omega$, in particular on $\partial^{\text{\rm hor}}\Omega_r\,$,
and $\Delta_{\alpha,\beta}f = -\varphi$ on $\Omega$. Thus, we get
$$
\int_{\Omega_r} \varphi(\mathfrak{w}) \, d\mathbf{m} 
= - \int_{\partial \Omega_r} 
\Bigl(\mathsf{n}^{\Omega_r}(\mathfrak{z})\,,\, \nabla_{\mathfrak{z}}^{\partial \Omega_r}f(\mathfrak{z})\Bigr)\,
d\mathbf{m}^{\partial \Omega_r}(\mathfrak{z}).
$$
We fix $\mathfrak{w} \in \Omega_r\,$. Continuing as in the proof of Lemma 
\ref{lem:poiss}, we see that
$$
\begin{aligned} 1 &= -\int_{\partial \Omega_r}  
\Bigl(\mathsf{n}^{\Omega_r}(\mathfrak{z})\,,\, \nabla_{\mathfrak{z}}^{\partial\Omega_r}
\mathbf{g}^{\Omega}(\mathfrak{z}, \mathfrak{w})\Bigr)\,d\mathbf{m}^{\partial \Omega_r}(\mathfrak{z})\\ 
&= \int_{\partial^{\text{\rm hor}} \Omega_r}\Pi^{\Omega}(\mathfrak{w},\mathfrak{z})\, 
d\mathbf{m}^{\partial \Omega_r}(\mathfrak{z}) \\
&\quad + \int_{N_{-r}} y^2 \partial_x \mathbf{g}^{\Omega}(\mathfrak{z}, \mathfrak{w}) \, 
d\mathbf{m}^{\partial \Omega_r}(\mathfrak{z})
+ \int_{N_r} -y^2 \partial_x \mathbf{g}^{\Omega}(\mathfrak{z}, \mathfrak{w}) \, 
d\mathbf{m}^{\partial \Omega_r}(\mathfrak{z})\,,
\end{aligned}$$
where $N_{-r}=N_{-r}(o)$ and $N_r=N_r(o)$ are the left and right hand vertical
boundary parts of $\Omega_r\,$.

Since we have fixed $\mathfrak{w}$, we can write
$\mathbf{g}^{\Omega}(\mathfrak{z}, \mathfrak{w}) = g(\mathfrak{z}) = g(x+\mathfrak{i}\, y, w)$, when 
$\mathfrak{z} = (x+\mathfrak{i}\, y, w)$ with $w \in T_v\,$. 
We observe that on $N_r$ for arbitrary $r \in \mathbb{R}$,  
$$
d\mathbf{m}^{\Omega_r}(\mathfrak{z}) 
= \underbrace{\beta^{\lceil \log_{\mathsf{q}}(y)\rceil} y^{\alpha}}_{\textstyle=:\phi(y)}
\,dy\,,\quad \text
{where}\quad \mathfrak{z} = (r + \mathfrak{i}\, y, w) \in N_r\,.
$$ 
The derivative $\partial_x$ refers to the $x$-coordinate of $\mathfrak{z}$. We can
consider the last measure, as well as the next integral below, as a measure and
integral on the neighbourhood star $N(o)$ in ${\mathbb T}$, because $N_r$ is an isometric
copy of that neighbourhood star.

Thus, for $r \in \mathbb{R}$, we can write 
$$
\int_{N_r} y^2 \partial_x \mathbf{g}^{\Omega}(\mathfrak{z}, \mathfrak{w}) \, 
d\mathbf{m}^{\partial \Omega_r}(\mathfrak{z}) = 
\int_{N(o)} y^2 \partial_x g(r+\mathfrak{i}\, y, w)\, \phi(y)\,dy.
$$
This notation is slightly inprecise; we keep in mind that in the integral 
over $N(o)$, one has to take the sum over such integrals, where each one
ranges over one of the edges of  $N(o)\,$. 

We have obtained that for any $x \in \mathbb{R} \setminus\{0\}$,
$$
\begin{aligned}
1 - \Pi(|x|) &= \int_{N(o)}\!\! y^2 \partial_x\Bigl(g(-|x|+\mathfrak{i}\, y, w) - g(|x|+\mathfrak{i}\, y, w)\Bigr)
\, \phi(y)\,dy\,,\quad\text{where}\\
\quad \Pi(|x|)  
&= \int_{\partial^{\text{\rm hor}} \Omega_{|x|}}\Pi^{\Omega}(\mathfrak{w},\mathfrak{z})\, d\mathbf{m}^{\partial \Omega_{|x|}}(\mathfrak{z})\,.
\end{aligned}
$$
Now we choose a fixed $\varepsilon > 0$ and let $r > \varepsilon$. We integrate both sides
over $x$ which varies in the interval 
$I = (r-\varepsilon\,,r+\varepsilon)$. Then we get, with an obvious exchange of the order of
integration
$$
\begin{aligned}
2\varepsilon &- \int_{I}\Pi(|x|)\,dx \\
&= \int_{N(o)} \int_{I} y^2 \partial_x\Bigl(g(-|x|+\mathfrak{i}\, y, w) 
- g(|x|+\mathfrak{i}\, y, w)\Bigr) \,dx\, \phi(y)\,dy \\
&= \int_{N(o)} \Bigl(g(-r - \varepsilon +\mathfrak{i}\, y, w) - g(r+\varepsilon+\mathfrak{i}\, y, w)\\
&\hspace{100pt} -g(-r+\varepsilon + \mathfrak{i}\, y,w)  + g(r-\varepsilon+\mathfrak{i}\, y, w)\Bigr)\, \phi(y)\,dy 
\end{aligned}
$$
When $r \to \infty$, the last integral 
tends to to $0$ by Lemma \ref{lem:zero}, while the one on the left hand side tends to 
$2\varepsilon \int_{\partial \Omega} \Pi^{\Omega}(\mathfrak{w},\cdot)\, d\mathbf{m}^{\partial \Omega}$.
Thus,
$$
2\varepsilon = 2\varepsilon \int_{\partial \Omega} \Pi^{\Omega}(\mathfrak{w},\mathfrak{z})\, d\mathbf{m}^{\Omega}(\mathfrak{z})\,,
$$
which yields the proposed statement.
\end{proof} 

We are finally in the position to extend Lemma \ref{lem:poiss} to
the unbounded set $\Omega = \Omega_o$ in ${\mathsf{HT}}$. This will also prove that
the exit measure $\mu$ from $\Omega$ has a continuous density with respect to
Lebesgue measure on $\partial \Omega$, or equivalently, with respect to 
Haar measure on $\mathcal{A}$.

\begin{thm}\label{thm:poiss}{\bf [Poisson representation]}
Let $h \in \mathcal D^2_{\alpha,\beta}(\overline\Omega)$
be harmonic, that is, $\Delta_{\alpha,\beta}h = 0$ in $\Omega$. 
Suppose that $h$ satisfies the growth condition
\begin{equation}\label{eq:expgro} 
|h(\mathfrak{z})| \le C_1 \,e^{C_2\, \mathsf{d}_{{\mathsf{HT}}}(\mathfrak{z},\mathfrak{o})} 
\end{equation}
on $\Omega$. ($C_1\,, C_2 > 0$.)

Then for every $\mathfrak{w} \in \Omega$,
$$
h(\mathfrak{w}) = \int_{\partial \Omega} \Pi^{\Omega}(\mathfrak{w},\mathfrak{z}) h(\mathfrak{z}) \, 
d\mathbf{m}^{\partial \Omega}(\mathfrak{z})\,.
$$
In particular, the Poisson kernel $\Pi^{\Omega}(\mathfrak{w},\cdot)$  
is the density of the exit distribution $\mu_{\mathfrak{w}}^{\Omega}$ 
with respect to the boundary measure $\mathbf{m}^{\partial \Omega}$, where $\mathfrak{w} \in \Omega$.
\end{thm}

\begin{proof} 
We first assume that $h$ is bounded.
Let $r$ be large enough such that
$\mathfrak{w} \in \Omega_r\,$. Then by Lemma \ref{lem:poiss} and because 
$d\mu_{\mathfrak{w}}^{\Omega_r} 
= \Pi^{\Omega_r}(\mathfrak{w},\cdot)\,d\mathbf{m}^{\partial \Omega_r}\,$,
$$
\begin{aligned}
h(\mathfrak{w}) &= \int_{\partial \Omega_r} \Pi^{\Omega_r}(\mathfrak{w},\mathfrak{z}) h(\mathfrak{z}) \, 
d\mathbf{m}^{\partial \Omega_r}(\mathfrak{z})\\
&= \int_{\partial^{\text{\rm hor}} \Omega_r} \Pi^{\Omega_r}(\mathfrak{w},\mathfrak{z}) h(\mathfrak{z}) \, 
d\mathbf{m}^{\partial \Omega}(\mathfrak{z})
+ \int_{\partial^{\text{\rm vert}} \Omega_{r}} h(\mathfrak{z}) d\mu_{\mathfrak{w}}^{\Omega_r}(\mathfrak{z})
\,.
\end{aligned}
$$
When $r \to \infty$, the last integral along $\partial^{\text{\rm vert}} \Omega_{r}$ 
tends to $0$ by Lemma \ref{lem:measure1} and boundedness of $h$. 
$$
\int_{\partial^{\text{\rm hor}} \Omega_r} \Pi^{\Omega_r}(\mathfrak{w},\mathfrak{z}) h(\mathfrak{z}) \, 
d\mathbf{m}^{\partial \Omega}(\mathfrak{z})
\to 
\int_{\partial \Omega} \Pi^{\Omega}(\mathfrak{w},\mathfrak{z}) h(\mathfrak{z}) \, 
d\mathbf{m}^{\partial \Omega}(\mathfrak{z})\,.
$$
Taking normal derivatives in \eqref{eq:bala}, we have for all $\mathfrak{w} \in \Omega_r$
and $\mathfrak{z} \in \partial\Omega$
$$
\Pi^{\Omega}(\mathfrak{w}, \mathfrak{z}) - \Pi^{\Omega_r}(\mathfrak{w}, \mathfrak{z})
= \int_{\partial^{\text{\rm vert}} \Omega_r} 
\Pi^{\Omega}(\cdot, \mathfrak{z}) \, d\mu_{\mathfrak{w}}^{\Omega_r} \ge 0
$$
with $\Pi^{\Omega_r}(\mathfrak{w}, \mathfrak{z}) = 0$ when $\mathfrak{z} \in \partial \Omega \setminus \partial^{\text{\rm hor}} \Omega_r\,$.
Therefore, using Proposition \ref{pro:one}, 
$$
\begin{aligned}
\bigg| \int_{\partial \Omega} &\Pi^{\Omega}(\mathfrak{w},\mathfrak{z}) h(\mathfrak{z})
\, d\mathbf{m}^{\partial \Omega}(\mathfrak{z})
- \int_{\partial \Omega} \Pi^{\Omega_r}(\mathfrak{w}, \mathfrak{z}) h(\mathfrak{z})
\, d\mathbf{m}^{\partial \Omega}(\mathfrak{z}) \bigg|\\ 
&= \bigg| \int_{\partial^{\text{\rm vert}} \Omega} \int_{\partial \Omega} 
\Pi^{\Omega}(\cdot, \mathfrak{z}) h(\mathfrak{z}) \, d\mathbf{m}^{\partial \Omega}(\mathfrak{z}) 
\, d\mu_{\mathfrak{w}}^{\Omega_r} \bigg|
\le \| h\|_{\infty}\, \mu_{\mathfrak{w}}^{\Omega_r} 
\bigl( \partial^{\text{\rm vert}}\Omega_r \bigr)\,,
\end{aligned}
$$
which tends to $0$ as $r \to \infty$ by Lemma \ref{lem:measure1}.

At this  point, we have obtained the Poisson representation for all
bounded harmonic functions. In the same way as we got Corollary \ref{cor:dens}, 
we get that
\begin{equation}\label{eq:mu-ex}
\Pi^{\Omega}(\mathfrak{w},\mathfrak{z})\,  d\mathbf{m}^{\partial \Omega}(\mathfrak{z}) = d\mu_{\mathfrak{w}}^{\Omega}(\mathfrak{z})\,,
\end{equation}
as proposed. 

\smallskip

Before we proceed, we make a simple observation on the metric. For any \hbox{$\mathfrak{z} = (z,w) \in {\mathsf{HT}}$,}
$$
\mathsf{d}_{{\mathsf{HT}}}(\mathfrak{z},\mathfrak{o}) \le 2 \log(|\operatorname{\text{\sl \!Re}} z| + 1) + \mathsf{d}_{{\mathbb T}}(w,o)\, \log \mathsf{q} .
$$
Indeed, this is the triangle inequality: let $\mathfrak{z}' = (\operatorname{\text{\sl \!Re}} z + \mathfrak{i}\,,o)$ be the point
on $\mathsf{L}_o$ with $\operatorname{\text{\sl \!Re}} \mathfrak{z}' = \operatorname{\text{\sl \!Re}} \mathfrak{z}$. Then $\mathsf{d}_{{\mathsf{HT}}}(\mathfrak{z},\mathfrak{z}') = \mathsf{d}_{{\mathbb T}}(w,o)\, \log \mathsf{q}$,
and $\mathsf{d}_{{\mathsf{HT}}}(\mathfrak{z}',\mathfrak{o}) = \mathsf{d}_{{\mathbb H}}(\operatorname{\text{\sl \!Re}} z + \mathfrak{i}\,, \mathfrak{i}\,) \le 2\log(|\operatorname{\text{\sl \!Re}} z| + 1)$,
a simple computation with the hyperbolic metric.
Thus,
\begin{equation}\label{eq:logbound}
\mathsf{d}_{{\mathsf{HT}}}(\mathfrak{z},\mathfrak{o}) \le 2 \log(|\operatorname{\text{\sl \!Re}} \mathfrak{z}| + 1) + D\,,
\quad \text{for all }\; \mathfrak{z} \in \overline\Omega\,,
\end{equation}
where $D =  \log \mathsf{q}\,$.

\smallskip

Now let $h$ be an arbitrary harmonic function on $\Omega$ that satisfies
\eqref{eq:expgro}. Suppose first that $\operatorname{\text{\sl \!Re}} \mathfrak{w} = 0$. Let $\sigma(0)=0$, and
as in the proof of Lemma \ref{lem:measure1}, let $\sigma(k)$ be the exit time
of Brownian motion from $\Omega_k$, where $k \in {\mathbb N}$. Then harmonicity of $h$ on $\Omega$
yields that $\bigl( h(X_{\sigma(k)}) \bigr)_{k \ge 0}$ is a discrete-time martingale.
Consider the random variable 
\begin{equation}\label{eq:Rmax}
R_{\max} = \max \{ |\operatorname{\text{\sl \!Re}} X_t| : t \le \tau\}.
\end{equation}
By Lemma \ref{lem:measure1},  we have
that 
$$
\mathsf{E}_{\mathfrak{z}}(R_{\max}^{\,s}) < \infty \quad \text{for every}\;  s> 0.
$$
Combining the growth condition \eqref{eq:expgro} with \eqref{eq:logbound},
we get 
$$
|h(X_t)| \le C_1\, e^{C_2 D}\, \bigl( R_{\max} + 1 \bigr)^{2 C_2} \quad \text{for } 0 \le t \le \tau\,,  
$$
an integrable upper bound. In particular, our martingale converges almost surely.

We now note that $\sigma(k)=\tau$ if $(X_t)$ exits $\Omega_k$ at $\partial^{\text{\rm hor}}\Omega_k$
and that $\sigma(k) < \tau$ if $(X_t)$ exits $\Omega_k$ at $\partial^{\text{\rm vert}}\Omega_k\,$.
In particular, $\sigma(k) = \tau$ when $k \ge R_{\max}\,$. 
We can decompose
$$
h(X_{\sigma(k)}) = h(X_{\tau})\,\mathbf{1}_{[\sigma(k)=\tau]} + 
h(X_{\sigma(k)})\,\mathbf{1}_{[\sigma(k)<\tau]}\,,
$$
and see that  
$h(X_{\sigma(k)}) \to h(X_\tau)$ almost surely. Using dominated convergence and
the martingale property, we conclude that
$$
f(\mathfrak{z}) = \mathsf{E}_{\mathfrak{z}}\bigl( h(X_{\tau}) \bigr)\,, 
$$
which proves the required statement for the case when $\operatorname{\text{\sl \!Re}} \mathfrak{z} =0$. For
arbirtary $\mathfrak{z} \in \Omega$, it follows immediately from the horizontal translation 
invariance of our Laplacian.
\end{proof}

\begin{cor}\label{cor:stren}
For any $\mathfrak{w} \in \Omega$, the exit measure $\mu_{\mathfrak{w}}^{\Omega}$ has continuous,
strictly positive density with respect to Lebesgue measure on the lines that
make up  $\partial\Omega$. It satisfies the moment condition
$$
\int_{\partial \Omega} \exp\bigl( \lambda \, |\operatorname{\text{\sl \!Re}} \mathfrak{z}|\bigr)\, d\mu_{\mathfrak{w}}^{\Omega}(\mathfrak{z}) < \infty
$$
for any $\lambda < \log(1/\rho)$ with $\rho$ as in Lemma \ref{lem:measure1}. Thus, we have the 
double exponential moment condition,
$$
\int_{\partial \Omega} \exp \Bigl(\exp\bigl(\mathsf{d}_{{\mathsf{HT}}}(\mathfrak{o},\mathfrak{z})/3\bigr)\Bigr)\, 
d\mu_{\mathfrak{w}}^{\Omega}(\mathfrak{z}) < \infty\,.
$$
\end{cor}

\begin{proof}
The fact that $\operatorname{\sf supp} \mu_{\mathfrak{w}}^{\Omega} = \partial \Omega$  follows from \eqref{eq:mu-ex},
or equivalently from the observation made in the proof of the next corollary that
all boundary points are regular: indeed, every regular point must belong to the support
of the exit distribution. 

The moment condition follows from Lemma \ref{lem:measure1}: if $(X_t)$ starts at $\mathfrak{w} \in \Omega_r$
and exits $\Omega$ at some $\mathfrak{z} \in \partial \Omega \setminus \partial \Omega_r$ then it must exit
$\Omega_r$ through its vertical boundary, that is
$$
\mu_{\mathfrak{w}}^{\Omega}(\partial \Omega \setminus \partial \Omega_r) 
\le \mu_{\mathfrak{w}}^{\Omega_r}\bigl( \partial^{\text{\rm vert}}\Omega_r \bigr)\,. 
\eqno\qedhere
$$
\end{proof}

\begin{thm}\label{thm:dir}
{\bf [Solution of the Dirichlet problem]}
Let $f \in \mathcal C(\partial\Omega)$ be 
such that the growth condition 
\eqref{eq:expgro} holds for $f$ on $\partial \Omega$. 
Then 
$$   
h(\mathfrak{w}) = \begin{cases}\displaystyle \int_{\partial \Omega} 
                    \Pi^{\Omega}(\mathfrak{w},\mathfrak{z}) f(\mathfrak{z}) \,
d\mathbf{m}^{\partial \Omega}(\mathfrak{z})\,, &\quad \mathfrak{w} \in \Omega\,,\\[10pt]
f(\mathfrak{w})\,,&\quad \mathfrak{w} \in \partial\Omega
\end{cases}
$$   
defines an extension of $f$ that is harmonic on $\Omega$ and continuous 
on $\overline\Omega$. 
Furthermore, there is at most one such extension which 
satisfies \eqref{eq:expgro} on the whole of $\Omega$.
\end{thm}

\begin{proof} 
A boundary point
$\mathfrak{z}$ of any open domain $\Omega \subset {\mathsf{HT}}$ is regular for the Dirichlet problem 
for continuous functions on $\partial \Omega$ if and only if 
the exit distribution
satisfies $\mathsf{Pr}_{\mathfrak{z}}[\tau^{\Omega} = 0] =1$ (a general fact from Potential
Theory). 

Every boundary point of $\Omega$ \emph{is} regular.
This follows from the fact that $\tau^{\Omega}$ (which was denoted $\sigma$ above)
is the same as the exit time
of the projected process $(Y_t)$ on $\mathbb{R}$ from the interval $[-1\,,\,1]$, see \eqref{eq:sig}.
But the Dirichlet problem for the latter interval (with boundary values at $\pm 1$)
is obviously solvable, as verified by
direct, elementary computations; see \cite[Lemma 4.4 \& proof]{BSSW2}.

Now, the function $h(\mathfrak{w})$ defined in the corollary is harmonic; its finiteness
is guaranteed by \eqref{eq:expgro} in combination with Corollary \ref{cor:stren}.  
By the same reason, it is continuous up to $\partial \Omega$, since every boundary point 
is regular, as we have just observed. Thus, we have a solution of the
Dirichlet problem.

Since $\Omega$ is not bounded, uniqueness is not immediate. 
(In the compact case, uniqueness would follow from the maximum principle: a 
harmonic function assumes its maximum on the boundary.) However, if some extension satisfies 
\eqref{eq:expgro} then we can apply Theorem \ref{thm:poiss} to see that it must coincide 
with the function defined in 
the present theorem.
\end{proof}

\begin{rmk}\label{rmk:density} 
What we have stated and proved for $\Omega$
also works in precisely the same way, with some more notational efforts, 
for more general sets which are unions of finitely many strips. 
We state several facts that will be used in the sequel.
\\[6pt]
(a) As already mentioned, group invariance implies that Theorem \ref{thm:poiss}, 
Theorem \ref{thm:dir} and Corollary \ref{cor:stren} are valid for any
set $\Omega_v$ ($v \in V({\mathbb T})$) in the place of $\Omega = \Omega_o\,$.
\\[6pt]
(b) We can take a compact metric subtree $T$ of ${\mathbb T}_{\mathsf{p}}$ which is \emph{full}
in the sense that for each of its vertices, either all, or else just one of its neigbours 
is in $T$. Then we can take the set 
$$
\overline \Omega_T = \{ \mathfrak{z} = (z,w) \in {\mathsf{HT}} : w \in T \},  
$$
and $\Omega_T$ is its interior.  By precisely the same methods as those
used above, with only minor and straightforward notational adaptations, one shows that 
also in this case, the exit distribution from $\Omega_T$ (with respect to any
starting point in that set) has a continuous
density with respect to the Lebesgue measure on the union of
the bifurcation lines that make up $\partial \Omega_T\,$. We also have the analogous 
Poisson representation and solution of the Dirichlet problem on $\Omega_T\,$, as well
as the exponential moment for the exit distribution,  which is supported by the
entire boundary of $\Omega_T\,$.
\\[6pt]
(c) For any set $\Theta \subset {\mathsf{HT}}$ and $\varepsilon > 0$, we let
\begin{equation}\label{eq:Lep}
\Theta^{\varepsilon} = \{ \mathfrak{z} \in \Theta : \mathsf{d}(\mathfrak{z}, \partial\Theta) > \varepsilon\}. 
\end{equation}
For all results stated in (b), we can also replace $\Omega_T$ by $\Omega_T^{\varepsilon}\,$,
where $0 < \varepsilon < \log \mathsf{q}\,$:
thus, $\partial \Omega_T^{\varepsilon}$ consists of horizontal lines in the interior of the boundary
strips of $\Omega_T\,$  at (hyperbolic) distance $\varepsilon$ from the boundary lines of $\Omega_T\,$.
\\[6pt]
(d) The same holds for ``sliced'' hyperbolic plane, interpreted as ${\mathsf{HT}}(\mathsf{q},1)$.
In this case, $\Omega_v$ should be replaced by the double strip 
$\widetilde \Omega_k = \pi^{{\mathbb H}}(\Omega_v)$ with $k = \mathfrak{h}(v)$,
and more generally this holds for any union of finitely many 
strips which is connected in ${\mathbb H}$, and we may extend those strips (or truncate) above and below 
up to some intermediate boundary lines. 
\\[6pt]
(e) In particular, all those results also hold for a single strip.
Note that this is in reality a classical situation because before reaching
the boundary of a strip, we are just observing ordinary hyperbolic Brownian motion with 
drift.  
\end{rmk}
The reason why we did not state any more general results, e.g. for relatively compact
domains that do not have rectangular shape, lies in the complication of how those sets
may cross (several times from different sides) the bifurcation lines. For sets with
sufficiently regular shape and piecewise smooth boundaries, such results can be stated and
elaborated with considerable further notational efforts, while the techniques remain
basically the same.

\section{Positive harmonic functions and the induced random walk}\label{sec:randomwalk}

We now finally turn our attention to positive harmonic functions on the whole of ${\mathsf{HT}}$.   
Recall the stopping times $\tau(n)$ from \eqref{eq:stop}. We assume that 
Brownian motion starts in some
point of $\operatorname{\sf LT}$, so that all increments $\tau(n) -\tau(n-1)$, $n \ge 1$, are i.i.d. and have
the same distribution as $\tau(1)$. Before coming to the core of this section,
we start with a simple result concerning the induced random walk $(W_{\tau(n)})_{n \ge 0}$ on 
${\mathbb T}$. An easy calculation \cite[Cor. 4.9]{BSSW2} shows that when it starts at a vertex of ${\mathbb T}$, then 
this is a nearest neighbour random walk on the vertex set of ${\mathbb T}$ with transition probabilities
\begin{equation}\label{eq:RWtree}
p_{{\mathbb T}}(v,v^-) = \frac{1}{1+\mathsf{a}} \quad\text{and}\quad p_{{\mathbb T}}(v^-,v) = \frac{\mathsf{a}}{(1+\mathsf{a})\mathsf{p}}\,,\quad v \in V({\mathbb T})
\end{equation}
with $\mathsf{a}$ as in \eqref{eq:rha}.
The one-step transition probabilities beween all other pairs of vertices are $0$.

\begin{lem}\label{lem:restrict-T}
\emph{(i)} If a non-negative function $f$ on ${\mathbb T}$
is $\Delta^{{\mathbb T}}_{\alpha,\beta}$-harmonic then its restriction to
$V({\mathbb T})$ is harmonic for the transition probabilites  of 
$(W_{\tau(n)})$. That is, for any 
$v \in V({\mathbb T})$
\begin{equation}\label{eq:restrict-T}
f(v) = \sum_{u \in V({\mathbb T})\,:\, u \sim v} p_{{\mathbb T}}(v,u)\,f(u)\,.
\end{equation}
\emph{(ii)}
Conversely, for any non-negative function $f$ on $V({\mathbb T})$ which satisfies 
\eqref{eq:restrict-T}, its unique extension to a  
$\Delta^{{\mathbb T}}_{\alpha,\beta}$-harmonic function on ${\mathbb T}$ is given on each edge 
$[v^-,v]$ of ${\mathbb T}$ and $w \in[v^-,v]$  by
\begin{equation}\label{eq:extend-T}
\begin{aligned}
f(w) &= \lambda(w)\cdot f(v^-) + \bigl(1 - \lambda(w)\bigr)\cdot f(v)\quad \text{with}\\[5pt]
\lambda(w) &=
\begin{cases} 
\dfrac{\mathsf{q}^{(\alpha-1)(\mathfrak{h}(v)-\mathfrak{h}(w))}-1}{\mathsf{q}^{(\alpha-1)}-1}\,,&\text{if}\;\alpha \ne 1\,,\\[5pt]
\mathfrak{h}(v)-\mathfrak{h}(w)\,,&\text{if}\;\alpha = 1\,,
\end{cases}
\end{aligned}
\end{equation}
where (recall)  $[v^-\,,\,v]$
corresponds to the real interval $[\mathfrak{h}(v)-1\,,\mathfrak{h}(v)]$.
\end{lem}

\begin{proof}
(i) is clear from Definition \ref{def:harmonic} of harmonic functions.

\smallskip

(ii) Given $f$ on $V({\mathbb T})$ and any vertex, we see from
\eqref{eq:bif-LapT} that  the extension to ${\mathbb T}$ within any edge $[v^-,v]$ must satisfy the
boundary value problem  
$\frac{1}{(\log \mathsf{q})^2} f'' +  \frac{\alpha-1}{\log \mathsf{q}} f' =0$,
with the given values $f(v^-)$ and $f(v)$ at the endpoints of the interval corresponding to 
that edge. Straightforward computation leads to the solution \eqref{eq:extend-T}.
It must satisfy the bifurcation condition of \eqref{eq:bif-LapT}, which can also be 
verified directly.  
\end{proof}

After this warmup, our main focus is on the induced process 
$(X_{\tau(n)})_{n\ge 0}$ on $\operatorname{\sf LT}$. We want to relate its positive harmonic functions with
the positive harmonic functions for our Laplacian $\Delta = \Delta_{\alpha,\beta}$ on ${\mathsf{HT}}$. 
Here, a function $f : \operatorname{\sf LT} \to \mathbb{R}$ is harmonic for that process if 
$f(\mathfrak{z}) = \mathsf{E}_{\mathfrak{z}}\bigl(f(X_{\tau(1)}\bigr)$ for any $\mathfrak{z} \in \operatorname{\sf LT}$. In view of the previous
section, this means that
\begin{equation}\label{eq:LTharmonic}
f(\mathfrak{z}) = \int_{\partial \Omega_v} f\, d\mu_{\mathfrak{z}}^{\Omega_v} \quad \text{for every } \;
v \in V({\mathbb T}) \; \text{ and } \mathfrak{z} \in \mathsf{L}_v\,.
\end{equation}
The analogue of Lemma \ref{lem:restrict-T} is by no means as straightforward; the reason is
that on ${\mathsf{HT}}$, the strips are non-compact so that one can make long sideways detours
in between two successive bifurcation lines. 

Let us define the probability measure 
\begin{equation}\label{eq:mu}
\mu = \mu_{\mathfrak{o}}^{\Omega} \quad \text{on} \quad \partial \Omega = \bigcup_{v: v \sim o} \mathsf{L}_v\,.
\end{equation}
Group invariance  yields that for any $\mathfrak{z} = (z,v) \in \operatorname{\sf LT}$  and $\mathfrak{g} \in \mathcal{A}$ with 
$\mathfrak{g} \mathfrak{o} = \mathfrak{z}$, we have the convolution identity
$\mu_{\mathfrak{z}}^{\Omega_v} = \delta_{\mathfrak{g}}*\mu$. 
Since $\mathcal{A}$ acts transitively on $\operatorname{\sf LT}$, the transition probabilities of 
the induced process are completely determined by $\mu$, so that we can consider
it as the random walk on $\operatorname{\sf LT}$ with law $\mu$. In view of this, we call the harmonic
functions of \eqref{eq:LTharmonic} the \emph{$\mu$-harmonic functions} on $\operatorname{\sf LT}$. They are
necessarily continuous, since $\mu$ has a continuous density. 

We note here that the measure $\mu$ is invariant under the stabiliser $\mathcal{A}_{\mathfrak{o}}$ of our 
reference point $\mathfrak{o}$ in $\mathcal{A}$. For sliced hyperbolic plane,
i.e., when $\mathsf{p}=1$, the group is $\operatorname{\sf Aff}({\mathbb H},\mathsf{q})$ and that stabiliser is trivial. When 
$\mathsf{p}  \ge 2$, the stabiliser is the compact subgroup of $\mathcal{A}$ consisting of all
$\mathfrak{g} = (\text{\sl id}_{{\mathbb H}},\gamma)$, where $\gamma \in \operatorname{\sf Aff}({\mathbb T})$ with $\gamma o = o$.

We also note the following fact that was anticipated in \cite{BSSW2}: by Theorem \ref{thm:poiss}, 
the measure $\mu$ has a continuous density
with respect to Lebesgue measure on the union of the boundary lines $\mathsf{L}_v\,$, $v \sim o$.
Also, by Lemma \ref{lem:measure1} it has exponential moment with respect to the 
\emph{Euclidean} distance on each of those lines. 

\smallskip

For positive harmonic functions, the following 
uniform \emph{Harnack inequality} is a consequence of \cite[Theorem 4.2]{BSSW1}.

\begin{pro}\label{pro:harnack}
For every non-negative $\Delta_{\alpha,\beta}$-harmonic function $h$ on (the whole of) ${\mathsf{HT}}$
and each $d>0$,
$$
h(\mathfrak{z}') \le C_d \,h(\mathfrak{z})\,,\quad \text{whenever}\quad \mathsf{d}_{{\mathsf{HT}}}(\mathfrak{z},\mathfrak{z}') \le d\,,
$$
where $C_d > 1$ is such that $C_d \to 1$ when $d \to 0$.
\end{pro}
In particular, every non-negative harmonic function must be strictly positive in
every point.
The next lemma will be important for understanding the relation between
$\Delta$-harmonic functions on ${\mathsf{HT}}$ and $\mu$-harmonic functions on $\operatorname{\sf LT}$.
It makes use of the straightforward extensions of Theorem \ref{thm:poiss} and 
Theorem \ref{thm:dir} to sets of the form $\Omega_T$ and $\Omega_T^{\varepsilon}\,$ for a 
full subtree $T$ of ${\mathbb T}$, as clarified in Remark \ref{rmk:density}.

\begin{lem}\label{lem:harnack-strip} Let $T$ be a full (compact) subtree of ${\mathbb T}$.
For every $d > 0$ and $0 < \varepsilon < \log \mathsf{q}$ there is
a constant $M_d = M_{d,T,\varepsilon} > 0$, with $M_d \to 1$ as $d \to 0$, such that for any positive harmonic function $h$ 
on $\Omega_T$,
$$
h(\mathfrak{z}') \le M_d\, h(\mathfrak{z})\,,\quad\text{whenever} \quad \mathfrak{z}, \mathfrak{z}' \in
\Omega_T^{\varepsilon}\quad\text{with}\quad
|\operatorname{\text{\sl \!Re}} \mathfrak{z}' - \,\operatorname{\text{\sl \!Re}} \mathfrak{z}| \le d\,.
$$
In particular, in this situation, the exit distributions from the set 
$\Omega_T$ satisfy
$\mu_{\mathfrak{z}'}^{\Omega_T} \le M_d\, \mu_{\mathfrak{z}}^{\Omega_T}\,.$
\end{lem}

\begin{proof} In analogy with the set $\Omega_{v,r}$ of \eqref{eq:Omv}, let us first take 
$\Omega_{T,r} = \{ \mathfrak{z} \in \Omega_T : |\operatorname{\text{\sl \!Re}} \mathfrak{z}| < r \}$ and its subset
$\Omega_{T,r}^{\varepsilon}\,$.
The intrinsic metric of $\Omega_T$ is equivalent with the naturally defined Euclidean 
metric of that set (i.e., the metric has to be extended across bifurcation lines using straight
geodesics). Then it follows once more from \cite[Theorem 4.2]{BSSW1} that every 
positive harmonic function on
$\Omega_{T,r}$ satisfies a Harnack inequality as proposed for $\mathfrak{z}, \mathfrak{z}'$ in 
$\Omega_{T,r}^{\varepsilon}\,$. Note that the whole of $\Omega_T$, resp. $\Omega_T^{\varepsilon}$ can be covered by 
translates of $\Omega_{T,r}\,$, resp. $\Omega_{T,r}^{\varepsilon}\,$.
By group invariance of our Laplacian, that inequality holds on each translate with
the same constants $M_d\,$.
\end{proof}

For the next Harnack inequality, we need some of the features of the geometry
of ${\mathsf{HT}}$, ${\mathbb T}$ and ${\mathbb H}$, as outlined above after Proposition \ref{pro:projections}.

\begin{pro}\label{pro:harnack-sliced} There are constants $C_1\,, C_2  > 0$ such that 
for every non-negative $\mu$-harmonic function $f$ on $\operatorname{\sf LT}$ and all
$\mathfrak{z}, \mathfrak{z}' \in \operatorname{\sf LT}$, 
$$
f(\mathfrak{z}') \le C_1\, e^{C_2 \,\mathsf{d}_{{\mathsf{HT}}}(\mathfrak{z},\mathfrak{z}')}\,f(\mathfrak{z})\,.
$$
\end{pro}

\begin{proof}
\emph{Step 1.}
Recall that $\Omega$ is the star of strips in $\operatorname{\sf LT}$ with middle line $\mathsf{L}_o\,$.
Since $\mu$ has a continuous density, the assumption of $\mu$-harmonicity yields that 
$f$ is continuous, and it is $\mu_{\mathfrak{z}}^{\Omega}$-integrable on 
$\partial \Omega$ for every $\mathfrak{z} \in \Omega$. Therefore 
$$
h(\mathfrak{z}) = \int_{\partial \Omega} f\, d\mu_{\mathfrak{z}}^{\Omega}
$$
defines a positive $\Delta$-harmonic extension of $f$ to the closure of $\Omega$. Since $f$ 
is $\mu$-harmonic, $h$ also coincides with $f$ on the middle line $\mathsf{L}_o$ of $\Omega$.
The function $h$ satisfies the Harnack inequality of Lemma \ref{lem:harnack-strip}. 
This yields the following for $h$ and consequently for $f$.
$$
\text {If}\quad \mathfrak{z}, \mathfrak{z}' \in \mathsf{L}_0  \quad\text{and}\quad |\operatorname{\text{\sl \!Re}} \mathfrak{z} - \,\operatorname{\text{\sl \!Re}} \mathfrak{z}'| 
\le 2\mathsf{q}^2 \quad
\text{then}\quad f(\mathfrak{z}') \le M\, f(\mathfrak{z})\,,
$$
where $M = M_{2\mathsf{q}^2}\,$. For any $v \in V({\mathbb T})$, we can map $\Omega$ to the star of 
strips $\Omega_v$ with middle line $\mathsf{L}_v$ by a mapping $\mathfrak{g} = (g,\gamma) \in \mathcal{A}$ 
such that $gz = \mathsf{q}^k\,z$ with $k = \mathfrak{h}(v)$, and $\gamma \in \operatorname{\sf Aff}({\mathbb T})$. It preserves 
both $\Delta$- and $\mu$-harmonicity. 
That mapping dilates horizontal Euclidean distances by the factor of $\mathsf{q}^k$. 
Therefore we also have the following statement for any $v \in V({\mathbb T})$ with the same constant 
$M$.
\begin{equation}\label{eq:local1}
\text {If}\quad \mathfrak{z}, \mathfrak{z}' \in \mathsf{L}_v \quad\text{and}\quad |\operatorname{\text{\sl \!Re}} \mathfrak{z} - \,\operatorname{\text{\sl \!Re}} \mathfrak{z}'| 
\le 2\mathsf{q}^{k+2} \quad \text{then}\quad f(\mathfrak{z}') \le M\, f(\mathfrak{z})\,.
\end{equation}

\smallskip
\noindent
\emph{Step 2.} We now use this for $V({\mathbb T}) \ni v \sim o$ and the associated point 
$\mathfrak{v} = (\mathfrak{i}\, \mathsf{q}^{\mathfrak{h}(v)}, v) \in \mathsf{L}_v \subset \partial \Omega$, an endpoint of
the ``star'' $N_0(o)$ as defined in \eqref{eq:hr-vr}. 
Let $I_v = \{ \mathfrak{z} \in \mathsf{L}_v : |\operatorname{\text{\sl \!Re}} \mathfrak{z}| \le 2\mathsf{q}^{2 + \mathfrak{h}(v)}\} = \mathsf{L}_{v,r}$ with 
$r = 2\mathsf{q}^{2 + \mathfrak{h}(v)}$. (Note that $\mathfrak{h}(v) = \pm 1$.)
Combining the assumption of $\mu$-harmonicity with \eqref{eq:local1},
$$
f(\mathfrak{o}) = \int_{\partial \Omega} f \, d\mu \ge \int_{I^v} f \, d\mu \ge
\mu(I_v) \, f(\mathfrak{v}) / M.
$$
By Corollary \ref{cor:stren}, $\mu(I_v) > 0$.
Re-using \eqref{eq:local1} on each $\mathsf{L}_v\,$, 
\begin{equation}\label{eq:local2}
\text {if}\quad \mathfrak{z} \in \mathsf{L}_o \cup \bigcup_{v \sim o} \mathsf{L}_v \quad\text{and}\quad 
|\operatorname{\text{\sl \!Re}} \mathfrak{z}| \le 2\mathsf{q} \quad \text{then}\quad f(\mathfrak{z}) \le \overline M\, f(\mathfrak{i}\,)\,,
\end{equation}
where $\overline M = M / \min \{\widetilde\mu(I_v) : v \sim o \}$.

\smallskip

\noindent
\emph{Step 3.}
Now we let $u \in V({\mathbb T})$ and $\mathfrak{z} \in \mathsf{L}_u$. Set $k = \mathfrak{h}(u)$, Then we can map $\mathfrak{i}\,$ to $\mathfrak{z}$
by a mapping $\mathfrak{g}=(g,\gamma) \in \mathcal{A}$, where
$g=\bigl(\begin{smallmatrix} \mathsf{p}^k & \operatorname{\text{\sl \!Re}} \mathfrak{z} \\ 0 & 1 \end{smallmatrix}\bigr) 
\in \operatorname{\sf Aff}({\mathbb H},\mathsf{q})$ and $\gamma \in \operatorname{\sf Aff}({\mathbb T})$ is such that $\gamma o = v$.
Thus, \eqref{eq:local2}  transforms into 
\begin{equation}\label{eq:local3}
\begin{gathered}
\text {If}\quad \mathfrak{z}, \mathfrak{z}' \in \mathsf{L}_u \cup \bigcup_{v \sim u} \mathsf{L}_v \quad\text{and}\quad 
|\operatorname{\text{\sl \!Re}} \mathfrak{z} - \,\operatorname{\text{\sl \!Re}} \mathfrak{z}'|  \le 2\mathsf{q}^{k+1} \\
\text{then}\quad 
f(\mathfrak{z}') \le \widetilde M\, f(\mathfrak{z})\,, \quad\text{where} \quad \widetilde M = 2\overline M.
\end{gathered}
\end{equation}

\smallskip

\noindent
\emph{Step 4.}
Now let $\mathfrak{z}=(z,v), \mathfrak{z}'=(z',v') \in \operatorname{\sf LT}$, and consider the situation where $v \curlywedge v' =v$,
that is, $v'$ lies ``above'' $v$ or is $=v$. Then we can choose some $\xi \in \partial^+{\mathbb T}$
such that $v \in \geo{\varpi\,\xi}$ and $\mathfrak{z}, \mathfrak{z}' \in {\mathbb H}_{\xi}\,$, as outlined in the
description of the metric \eqref{eq:metric} of ${\mathsf{HT}}$. A geodesic arc $\geo{\mathfrak{z}\,\mathfrak{z}'}$ (depending
on $\xi$) is then
given by the (isometric) image in ${\mathbb H}_{\xi}$ of the geodesic arc $\geo{z\,z'}$ in 
${\mathbb H}$.\footnote{This explains that geodesics in ${\mathsf{HT}}$ are in general not unique, because there
are different feasible choices for $\xi$ that may lead to different geodesics, according
to the mutual position of $z$ and $z'$ in ${\mathbb H}$.} 
The arc $\geo{\mathfrak{z}\,\mathfrak{z}'}$ has a highest point 
$\tilde{\mathfrak z} = (\tilde z, \tilde w)$, 
where $\tilde z = z \wedge z'$.
Our arc may cross more than one strip. It meets $\operatorname{\sf LT} \cap \,{\mathbb H}_{\xi}$
in successive points $\mathfrak{z}_j = (z_j,v_j)$, $j=0,\dots,n$, where $\mathfrak{z}_0=\mathfrak{z}$ and $\mathfrak{z}_n = \mathfrak{z}'$.
All $v_j$ are vertices on the ray $\geo{v\,\xi}$ in ${\mathbb T}$. Except possibly for one ``top'' sub-arc
that contains $\tilde{\mathfrak z}$, any $\geo{\mathfrak{z}_{j-1}\,\mathfrak{z}_j}$ crosses one strip: in the initial,
ascending part (if it is present), $v_j^- = v_{j-1}$ and that strip is $S_{v_j}\,$, while in the
terminal, descending part (if present),  $v_j = v_{j-1}^-$ and that strip is $S_{v_{j-1}}\,$.
In both cases (i.e. except possibly for the top arc), 
$\mathsf{d}_{{\mathsf{HT}}}(\mathfrak{z}_{j-1}\,,\mathfrak{z}_j) \ge \log \mathsf{q}$ (the distance between any two adjacent bifurcation lines),
and we get
$$
\mathsf{d}_{{\mathsf{HT}}}(z,z') \ge (n-1) \log\mathsf{q}\,.
$$
For any $j \in \{ 1, \dots, n\}$, let $k(j) = \mathfrak{h}(v_{j-1} \curlywedge v_{j})$ be the level
(horocycle number) of the lower boundary line of the strip in which the $j$-th sub-arc  
is contained. Therefore, the ${\mathbb H}$-coordinates of $\mathfrak{z}_{j-1}$ and $\mathfrak{z}_j$ satisfy
$\operatorname{\text{\sl \!Im}} z_{j-1}\wedge z_j \le \mathsf{q}^{k(j)+1}$.
Thus by \eqref{eq:RE-IM} 
$$
|\operatorname{\text{\sl \!Re}} \mathfrak{z}_j- \, \operatorname{\text{\sl \!Re}} \mathfrak{z}_{j+1}| = |\operatorname{\text{\sl \!Re}} z_j- \, \operatorname{\text{\sl \!Re}} z_{j+1}| \le 2\mathsf{q}^{k(j)+1}\,,
$$
and \eqref{eq:local3} yields that $f(\mathfrak{z}_{j+1}) \le \widetilde M\, f(\mathfrak{z}_j)$. 
Putting the sub-arcs together,
$$
f(\mathfrak{z}') \le \widetilde M^n\, f(\mathfrak{z}) \le 
\exp \Bigl( C_2 \,\bigl(\mathsf{d}_{{\mathsf{HT}}}(\mathfrak{z},\mathfrak{z}')+1\bigr)\Bigr)\, f(\mathfrak{z})\,, 
\quad\text{where}\quad C_2= \frac{\log \widetilde M}{\log\mathsf{q}}
$$ 
\smallskip
\noindent
\emph{Step 5.} At last, let $\mathfrak{z}=(z,v), \mathfrak{z}'=(z',v') \in \operatorname{\sf LT}$ be such that 
$v \curlywedge v' \notin \{ v, v'\}$. Then \eqref{eq:metric} shows that any geodesic arc
$\geo{\mathfrak{z}\, \mathfrak{z}'}$ in ${\mathsf{HT}}$ decomposes into two sub-arcs $\geo{\mathfrak{z}\, \mathfrak{z}''}$
and $\geo{\mathfrak{z}''\, \mathfrak{z}'}$, each of which is as in Step~4. Therefore
$$
f(\mathfrak{z}') \le \exp\Bigl( C_2 \,\bigl(\mathsf{d}_{{\mathsf{HT}}}(\mathfrak{z}',\mathfrak{z}'')+1\bigr)\Bigr)\, f(\mathfrak{z}'')
\le \exp\Bigl( C_2 \,\bigl(\mathsf{d}_{{\mathsf{HT}}}(\mathfrak{z}',\mathfrak{z})+2\bigr)\Bigr)\, f(\mathfrak{z})
$$
This concludes the proof.
\end{proof}

Now consider a set $\Omega_T\,$, where $T$ is a full subtree of ${\mathbb T}$,
as defined in Remark \ref{rmk:density}(b). 
We note the following facts.

\begin{rmk}\label{rmk:exit}
If the starting point $\mathfrak{z}$ of $(X_t)$ lies on
a line $\mathsf{L}_v$ in the interior of $\Omega_T$ (that is, $v$ is an interior vertex of $T$)
then the exit distribution $\mu_{\mathfrak{z}}^{\Omega_T}$ of $(X_t)$ from $\Omega_T$
is the same as the exit distribution of the random walk $(X_{\tau(n)})$ from $\Omega_T\,$. 

In view of Remark
\eqref{rmk:density}, it has the same properties as those of lemmas \ref{lem:measure1} 
(with different $\rho$) and \ref{lem:zero}, as well as Proposition \ref{pro:one} 
and corollaries \ref{cor:stren} and \ref{thm:dir}. 
\end{rmk}

Then we have the following (easier) analogue of Theorem \ref{thm:poiss}.

\begin{pro}\label{pro:poisson}
Let $f$ be a $\mu$-harmonic function on $\overline\Omega_T\,$,
that is, it satisfies \eqref{eq:LTharmonic} on every line $\mathsf{L}_v \subset \Omega_T\,$.
Suppose furthermore that $f$ 
satisfies the growth condition
\begin{equation}\label{eq:expgro-f} 
|f(\mathfrak{z})| \le C_1 \,e^{C_2\, \mathsf{d}_{{\mathsf{HT}}}(\mathfrak{z},\mathfrak{o})} 
\end{equation}
on $\partial\Omega_T\,$, where $C_1\,, C_2 > 0$.
Then for every $\mathfrak{z} \in \Omega_T \cap \operatorname{\sf LT}$,
$$
f(\mathfrak{z}) = \int_{\partial \Omega_T} f \, d \mu_{\mathfrak{z}}^{\Omega_T}\,.
$$
\end{pro}

\begin{proof}
We first observe that \eqref{eq:logbound} from the proof of Theorem \ref{thm:poiss} 
also holds on $\overline \Omega_T$ with constant 
$D = \max \{ \mathsf{d}_{{\mathbb T}}(w,o) : w \in T \}\cdot \log \mathsf{q}\,$.

This time, let $\tau = \tau^{\Omega_T}$ be the exit time from $\Omega_T\,$.
Let the starting point be $\mathfrak{z} \in \Omega_T \cap \operatorname{\sf LT}\,$.
Consider the process $X_n^T = X_{\min\{\tau(n), \tau\}}$ on
$\overline\Omega_T \cap \operatorname{\sf LT}\,$. This is the induced random walk on $\operatorname{\sf LT}$
stopped upon reaching $\partial \Omega_T\,$.
Then $\bigl(f(X_n^T)\bigr)_{n \ge 0}$ is a martingale by $\mu$-harmonicity of~$f$.
Note that $\tau = \tau(\boldsymbol{n})$ for some random $\boldsymbol{n}$.
Thus, for non-random $n$ tending to $\infty$, 
$$
f(X_n^T) = f(X_{\tau})\,\mathbf{1}_{[\tau \le \tau(n)]} + f(X_{\tau(n)})\,\mathbf{1}_{[\tau > \tau(n)]}
\to f(X_{\tau})\quad\text{almost surely.}
$$
As in \eqref{eq:Rmax}, we consider the random variable 
$R_{\max} = \max \{ |\operatorname{\text{\sl \!Re}} X_t| : t \le \tau\}$ (only that this time, $\tau$ is more general). 
By Lemma \ref{lem:measure1}, resp. its variant for $\Omega_T\,$, we have again that it
has finite moments of all orders.
Combining the growth condition \eqref{eq:expgro-f} with \eqref{eq:logbound},
we get 
$$
|f(X_n^T)| \le C_1\, e^{C_2 D}\, \bigl( R_{\max} + 1 \bigr)^{2 C_2}\,,  
$$
an integrable upper bound. Thus, by dominated convergence and the Martingale 
property, $f(\mathfrak{z}) = \mathsf{E}_{\mathfrak{z}}\bigl(f(X_{\tau})\bigr)$, which is just another form
of the proposed formula.
\end{proof}

Now we can formulate and prove the main result of this section, which is the
analogue of Lemma \ref{lem:restrict-T} for ${\mathsf{HT}}$.

\begin{thm}\label{thm:restrict-HT}
\emph{(i)} If a non-negative function $h$ on ${\mathsf{HT}}$
is $\Delta_{\alpha,\beta}$-harmonic then its restriction to
$\operatorname{\sf LT}$ is $\mu$-harmonic. 
\\[5pt]
\emph{(ii)}
Conversely, for any non-negative $\mu$-harmonic function $f$ on $\operatorname{\sf LT}$, its unique 
extension to a  
$\Delta_{\alpha,\beta}$-harmonic function on ${\mathsf{HT}}$ is given in each open 
strip $\mathsf{S}_v^o$ of ${\mathsf{HT}}$ by
\begin{equation}\label{eq:extend-H}
h(\mathfrak{z}) = \mathsf{E}_{\mathfrak{z}}\bigl(f(X_{\tau(1)})\bigr) 
= \int_{\partial \mathsf{S}_v} f\, d\mu_{\mathfrak{z}}^{\mathsf{S}_v^o}\,,
\end{equation}
where (recall) $\mu_{\mathfrak{z}}^{\mathsf{S}_v^o}$ is the exit distribution from 
$\mathsf{S}_v^o$ of $(X_t)$ starting at 
$\mathfrak{z} \in \mathsf{S}_v^o\,$.
\end{thm}

\begin{proof}
(i) If $h$ is non-negative
and is $\Delta_{\alpha,\beta}$-harmonic on ${\mathsf{HT}}$ 
then it satisfies the Harnack
inequality \ref{pro:harnack}. Therefore, it also satisfies the growth condition
\eqref{eq:expgro}, and Theorem \ref{thm:poiss} applies. 
That theorem provides (i) for points
$\mathfrak{z} \in \mathsf{L}_o$. By transitivity of the action of the group $\mathcal{A}$ on $\operatorname{\sf LT}$
as well as group invariance of the Laplacian, this holds  along any line 
$\mathsf{L}_v\,$.

\smallskip 

(ii) We consider the following exhaustion of ${\mathbb T}$ by finite, full subtrees $T(k)$.
Let $o_k$ be the $k$-th predecessor of the root $o$ of ${\mathbb T}$, that is,
the element on $\geo{o\,\varpi}$ at distance $k$ from $o$. 
Let $W(k)$ be the set of all vertices on the horocycle $H_k$ at distance $2k-1$
from $o_{k-1}$. (That is, they have $o_{k-1}$ as their $(2k-1)$-st predecessor.)
Then $T(k)$ is the union of all geodesic segments $\geo{o_k\, w}$, where $w \in W(k)$.
We consider $\Omega_{T(k)}\,$, with
$$
\partial \Omega_{T(k)} = \mathsf{L}_{o_k} \cup \bigcup_{w \in W(k)} \mathsf{L}_w\,.
$$
Now let $f$ be a positive $\mu$-harmonic function on $\operatorname{\sf LT}$. Then it is continuous
(since $\mu$ has a continuous density by Theorem \ref{thm:poiss}) and satisfies \eqref{eq:expgro-f}
by Proposition \ref{pro:harnack-sliced}. In particular, we can apply Theorem \ref{thm:dir}, resp., its
extended version according to Remark \ref{rmk:density}  
to the restriction of $f$ to $\partial \Omega_{T(k)}\,$, and
$$
h_k(\mathfrak{z}) = \int_{\partial \Omega_{T(k)}} f \, d\mu_{\mathfrak{z}}^{\Omega_{T(k)}}\,,\quad \mathfrak{z} \in 
\Omega_{T(k)}
$$ 
defines a harmonic function on $\Omega_{T(k)}$ which is continuous up to
the boundary values of $f$. By Proposition \ref{pro:poisson}, $h_k = f$ 
on every bifurcation line $\mathsf{L}_v$ with $v \in V\bigl(T(k)\bigr)$.

Now consider $h_{k+1}\,$: to this function, we can apply  
Lemma \ref{lem:harnack-strip} on the set
$\overline\Omega_{{\mathbb T}(k)}\! \subset \Omega_{{\mathbb T}(k+1)}^{\varepsilon}$: there is a constant
$M(k)$ such that $h_{k+1}(\mathfrak{z}) \le M(k) h_{k+1}(\tilde{\mathfrak z})$ for every 
$\mathfrak{z} \in \overline\Omega_{{\mathbb T}(k)}$, where $\tilde{\mathfrak z}$ is the element on $\mathsf{L}_o$
with $\operatorname{\text{\sl \!Re}} \tilde{\mathfrak z} = \operatorname{\text{\sl \!Re}} \mathfrak{z}$. But on $\mathsf{L}_o\,$, the function $h_{k+1}$
coincides with $f$, which satisfies the growth condition \eqref{eq:expgro-f}
by Proposition \ref{pro:poisson}. Therefore also $h_{k+1}$ satisfies 
\eqref{eq:expgro} on the whole of $\overline\Omega_{{\mathbb T}(k)}$, with possibly 
modified constants. But now we can apply the uniqueness statement of 
Theorem \ref{thm:dir}, which yields that on $\overline\Omega_{{\mathbb T}(k)}$,
the function $h_{k+1}$ is the Poisson integral of $f$ taken over $\partial\Omega_{{\mathbb T}(k)}\,$.

In other words, the restriction of $h_{k+1}$ to $\overline\Omega_{{\mathbb T}(k)}$ is $h_k\,$.
If $\mathfrak{z} \in {\mathsf{HT}}$ then there is $k$ such that 
$\mathfrak{z} \in \Omega_{T(k)}\,$, and setting $h(\mathfrak{z}) = h_k(\mathfrak{z})$, we get a well-defined,
positive $\Delta$-harmonic function on ${\mathsf{HT}}$ that coincides with $f$ on $\operatorname{\sf LT}$.

Uniqueness is straightforward: if $\tilde h$ is any other positive $\Delta$-harmonic 
function that extends $f$, then by Proposition \ref{pro:harnack} it satisfies 
the growth condition \eqref{eq:expgro} on the whole of ${\mathsf{HT}}$. We can apply once more the uniqueness
statement of Theorem \ref{thm:dir} on any $\Omega_{T(k)}$ and get that $\tilde h$
concides with $h_k=h$ on the latter set, for any $k$.  

\smallskip
  
The statement (ii) itself now follows a posteriori.
\end{proof}

\begin{rmk}\label{rmk:Hb}
The results of this section also apply to the Laplacian $\Delta_{\alpha,\beta}^{{\mathbb H}}$ on 
${\mathbb H}(\mathsf{q}) = {\mathsf{HT}}(\mathsf{q},1)$, where the metric tree is $\mathbb{R}$ with vertex set ${\mathbb Z}$. In this case,
$\operatorname{\sf LT}$ is better denoted  $\operatorname{\sf L}\!{\mathbb H} = \bigcup_k \mathsf{L}_k \subset {\mathbb H}$, the union of all bifurcation lines 
in upper half plane. Note that $\operatorname{\sf Aff}({\mathbb H},\mathsf{q})$ acts simply transitively on $\operatorname{\sf L}\!{\mathbb H}$,
so that this group can be identified with $\operatorname{\sf L}\!{\mathbb H}$. The random walk $(Z_{\tau(n)})$
is governed the the probability measure $\widetilde \mu$ on $\mathsf{L}_1 \cup \mathsf{L}_{-1}$
which is the exit distribution of $(Z_t)$ from the double strip $\mathsf{S}_0 \cup \mathsf{S}_1$
when $Z_0 = \mathfrak{i}\,$. 

When we start with $\Delta_{{\mathsf{HT}}}^{\alpha,\beta}\,$ on ${\mathsf{HT}}(\mathsf{q},\mathsf{p})$, then in view of Proposition 
\ref{pro:projections} we need to consider $\Delta_{\alpha,\beta\mathsf{p}}^{{\mathbb H}}$ here. In this
situation, $\widetilde \mu$ is the image of $\mu$ under the projection $\pi^{{\mathbb H}}$.
What we have proved for $\mu$ on ${\mathsf{HT}}$ is also true for $\widetilde \mu$ on ${\mathbb H}$.
Of course, we shall call the associated harmonic functions on $\operatorname{\sf L}\!{\mathbb H} \equiv \operatorname{\sf Aff}({\mathbb H},\mathsf{q})$ the 
$\widetilde\mu$-harmonic functions.
\end{rmk}

\section{Decomposition of positive harmonic functions}\label{sec:decomp}

The aim of this section is to provide tools for giving a complete description
of all positive $\Delta$-harmonic functions on the whole of ${\mathsf{HT}}$. 
For this purpose, we shall need an understanding of the geometric boundary
at infinity of ${\mathsf{HT}}$. We quickly review the basic facts; a detailed
description can be found in \cite[Section~5]{BSSW2}. 

The boundary (space of ends) $\partial {\mathbb T}$ was described in the lines after
\eqref{eq:RE-IM}. The end compactification is $\widehat{\mathbb T} = {\mathbb T} \cup \partial T$. The topology
is such that a sequence $(y_n)$ converges to a point $y$ in the metric tree ${\mathbb T}$ if it 
converges with respect to the original metric, while it converges to an end $\xi \in \partial T$ 
if the geodesics $\geo{o\,z_n}$ and $\geo{o,\xi}$ share initial pieces whose graph
length tends to infinity. The boundary of ${\mathbb H}(\mathsf{q})$, which is just hyperbolic half plane
metrically, is $\partial{\mathbb H} = \mathbb{R} \cup \{\boldsymbol{\infty}\}$. The topology of 
$\widehat{\mathbb H}= {\mathbb H} \cup \partial {\mathbb H}$ is the classical hyperbolic compactification; it is maybe
better understood when one considers the Poincar\'e disk model instead of upper half plane,
where $\widehat{\mathbb H}$ is the Euclidean closure (the closed unit disk), and $\boldsymbol{\infty}$ corresponds
to the north pole, while the remaining part of the unit circle corresponds to the lower
boundary line $\mathbb{R}$ of upper half plane. 

Now ${\mathsf{HT}}(\mathsf{q},\mathsf{p})$ is a subspace of the direct product ${\mathbb H}(\mathsf{q}) \times {\mathbb T}_{\mathsf{p}}\,$, so that
the most natural geometric compactification $\widehat{\mathsf{HT}}$ of ${\mathsf{HT}}$ is its closure
in $\widehat {\mathbb H} \times \widehat {\mathbb T}$. As mentioned, a detailed description of the resulting
boundary $\partial\!{\mathsf{HT}}$ and of the different ways of convergence to the boundary can be
found in \cite{BSSW2}.

\smallskip

Transience of $(X_t)$ implies that the \emph{Green kernel} 
\begin{equation}\label{eq:greenkernel}
\mathbf{g}_{\alpha,\beta}(\mathfrak{w},\mathfrak{z}) = \mathbf{g}(\mathfrak{w},\mathfrak{z}) =\int_0^\infty 
\mathbf{h}_{\alpha,\beta}(t,\mathfrak{w},\mathfrak{z})\,dt \quad 
(\mathfrak{w}, \mathfrak{z} \in {\mathsf{HT}}\,,\;\mathfrak{w} \ne\mathfrak{z}) 
\end{equation}
is  strictly positive and finite, compare with \cite{BSSW1}. The function
$\mathfrak{w} \mapsto \mathbf{g}(\mathfrak{w},\mathfrak{z})$ is harmonic on ${\mathsf{HT}} \setminus \{\mathfrak{z}\}$.
Along with Proposition \ref{pro:harnack},  the next uniform Harnack inequalities
for the Green kernel are again a consequence of \cite[Theorem 4.2]{BSSW1}.

\begin{pro}\label{pro:harnack-green}
With the same constants $C_d\,$ as in Proposition \ref{pro:harnack}  
(in particular, $C_d \to 1$ when $d \to 0$)
$$
\frac{\mathbf{g}(\mathfrak{w},\mathfrak{z}')}{\mathbf{g}(\mathfrak{w},\mathfrak{z})} \le C_d \quad\text{and}\quad
\frac{\mathbf{g}(\mathfrak{z}',\mathfrak{w})}{\mathbf{g}(\mathfrak{z},\mathfrak{w})} \le C_d \,,
$$
whenever $\mathsf{d}_{{\mathsf{HT}}}(\mathfrak{z},\mathfrak{z}') \le d$ and 
$\min \{ \mathsf{d}_{{\mathsf{HT}}}(\mathfrak{z},\mathfrak{w}), \mathsf{d}_{{\mathsf{HT}}}(\mathfrak{z}',\mathfrak{w}) \} \ge 10 (d+1)$.
\end{pro}

Positive harmonic functions can be described via the \emph{Martin boundary.}
The regularity properties of our Green kernel allow us to use the following
approach, which is a special case of the potential theory (more precisely, 
Brelot theory) outlined in the book of
{\sc Constantinescu and Cornea}~\cite[Chapter 11]{CC}. See also
{\sc Brelot}~\cite{Bre}, and  
in the Markov chain setting, compare with {\sc Revuz}~\cite[Chapter 7]{Rev}. 

We fix the ``origin'' $\mathfrak{o} = (\mathfrak{i}\,,o)$ of ${\mathsf{HT}}$ as the reference point.
A function $f \in \mathcal{H}^+$ is called \emph{minimal} if $f(\mathfrak{o})=1$ and
whenever $f_1 \in \mathcal{H}^+$ is such that $f \ge f_1$ then $f_1/f$ is constant.
The minimal harmonic functions are the extremal points of the convex set
$\{ f \in {\mathsf{HT}}^+ : f(\mathfrak{o})=1 \}$, which is a base of the convex cone
${\mathsf{HT}}^+$ of all positive harmonic functions. Every positive harmonic function is a 
convex combination (with respect to a Borel measure) of minimal ones. 
To make this more precise, define the \emph{Martin kernel} by
\begin{equation}
k(\mathfrak{w},\mathfrak{z}) = \frac{\mathbf{g}(\mathfrak{w},\mathfrak{z})}{\mathbf{g}(\mathfrak{o},\mathfrak{z})}
\end{equation}
Then the Martin compactification is the unique (up to homeomorphism)
minimal compactification of the underlying space (i.e., ${\mathsf{HT}}$) such that
all functions $k(\mathfrak{w},\cdot)$, $\mathfrak{w} \in {\mathsf{HT}}$, extend continuously.
The extended kernel is also denoted by $k(\cdot,\cdot)$.
The \emph{Martin boundary} $\mathcal{M}=\mathcal{M}(\Delta)$ is the set of ideal boundary points
in this compactification (i.e., the new points added ``at infinity'').
It is compact and metrizable.

From the general theory \cite{CC}, it is known that 
every minimal harmonic 
function arises as a boundary function $k(\cdot,\boldsymbol{\xi})$, where $\boldsymbol{\xi} \in \mathcal{M}$.
The (Borel) set $\mathcal{M}_{\min}$ of all $\boldsymbol{\xi} \in \mathcal{M}$ for which $k(\cdot,\boldsymbol{\xi})$
is minimal is called the \emph{minimal Martin boundary.}
The \emph{Poisson-Martin} representation theorem says that for every 
$h \in \mathcal{H}^+$ there is a unique Borel measure $\nu^h$ on $\mathcal{M}$ such that
$\nu(\mathcal{M} \setminus \mathcal{M}_{\min}) = 0$ and
\begin{equation}\label{eq:PoisMart}
h = \int_{\mathcal{M}} k(\cdot,\boldsymbol{\xi})\,d\nu^h(\boldsymbol{\xi})\,.
\end{equation}
If $h^{{\mathbb H}}$ is an $\Delta^{{\mathbb H}}$-harmonic function on ${\mathbb H}$ then by Proposition 
\ref{pro:projections}, the function $h^{{\mathbb H}}\circ \pi^{{\mathbb H}}$ 
is $\Delta$-harmonic on ${\mathsf{HT}}$, and the analogous property holds when we lift a
$\Delta^{{\mathbb T}}$-harmonic function $h^{{\mathbb T}}$ from ${\mathbb T}$ to ${\mathsf{HT}}$. 
Combining both cases, the function $h(\mathfrak{z}) = h^{{\mathbb H}}(z) + h^{{\mathbb T}}(w)$ 
(where $\mathfrak{z} = (z,w) \in {\mathsf{HT}}$) is also $\Delta$-harmonic on ${\mathsf{HT}}$.

The first main purpose of this section and the machinery set up so far is to prove
the following.

\begin{thm}\label{thm:decompose}
Every positive $\Delta_{\alpha,\beta}$-harmonic $h$ on ${\mathsf{HT}}$ has the form
$$
h(\mathfrak{z}) = h^{{\mathbb H}}(z) + h^{{\mathbb T}}(w)\,, \quad \mathfrak{z} = (z,w) \in {\mathsf{HT}}\,,
$$
where $h^{{\mathbb H}}$ is non-negative $\Delta^{{\mathbb H}}_{\alpha,\beta\mathsf{p}}$-harmonic 
on ${\mathbb H}$ and $h^{{\mathbb T}}$ is non-negative $\Delta^{{\mathbb T}}_{\alpha,\beta}$-harmonic on ${\mathbb T}$.   
\end{thm}

For the proof, we shall use the crucial fact that $\Delta_{\alpha,\beta}$ 
is invariant under the group $\mathcal{A} = \mathcal{A}(\mathsf{q},\mathsf{p})$ defined in Theorem \ref{thm:isogroup},  
that is, $\mathfrak{t}_{\mathfrak{g}}(\Delta_{\alpha,\beta} f) = \Delta_{\alpha,\beta} (\mathfrak{t}_{\mathfrak{g}} f)$ 
for all $f \in \text{\rm Dom}(\Delta_{\alpha,\beta})$ and
all $\mathfrak{g} \in \mathcal{A}$, where 
\begin{equation}\label{eq:translation}
\mathfrak{t}_{\mathfrak{g}} f(\mathfrak{z}) = f(\mathfrak{g}\mathfrak{z}).
\end{equation}
 Let
$$
G(\mathfrak{w},\mathfrak{z}) = \mathbf{g}(\mathfrak{w},\mathfrak{z})\phi_{\alpha,\beta}(\mathfrak{z})\,,
$$
where $\phi_{\alpha,\beta}$ is as in \eqref{eq:mab}.  
We have $k(\mathfrak{w},\mathfrak{z}) = G(\mathfrak{w},\mathfrak{z})/G(\mathfrak{o},\mathfrak{z})$, 
and it follows from \eqref{eq:groupinvariance} that 
$$
G(\mathfrak{g}\mathfrak{w},\mathfrak{g}\mathfrak{z}) = G(\mathfrak{w},\mathfrak{z}) 
\quad\text{for all}\quad \mathfrak{w}, \mathfrak{z} \in {\mathsf{HT}} \;(\mathfrak{w} \ne \mathfrak{z})\; \text{and}\; 
\mathfrak{g} \in \mathcal{A}\,.
$$
We shall use specific subgroups of
$\mathcal{A}$. First, in the tree, consider once more the $k$-th predecessor $o_k$ of the root
vertex. Thus, $o_{k} \in H_{-k}\,$. Now define
$$
\mathcal{B}({\mathbb T}) = \{ \mathfrak{g}_{\,\psi\,} = (\text{\sl id}_{{\mathbb H}},\psi) : 
\psi \in \operatorname{\sf Aff}({\mathbb T})\,,\; \psi o_{k} = o_{k}\;\text{for some}\; k \in {\mathbb N} \}\,.
$$
This is the embedding into $\mathcal{A}$ of the \emph{horocyclic subgroup} 
of $\operatorname{\sf Aff}({\mathbb T})$, which consists of all isometries of the tree that leave each
horocycle invariant. The latter is the analogue of the group of all ``horizontal''
translations in $\operatorname{\sf Aff}({\mathbb H},\mathsf{q})$, that is, of all mappings $g_b = (1,b)$,
where $b \in \mathbb{R}$, acting on ${\mathbb H}$ by $z \mapsto z+b$. Thus, we define
$$
\mathcal{B}({\mathbb H}) = \{ \mathfrak{g}_{\,b\,} = (g_b,\text{\sl id}_{{\mathbb T}}) : b \in \mathbb{R} \}\,,
$$
which is of course isomorphic with the additive group $\mathbb{R}$ acting on ${\mathbb H}$.

\begin{thm}\label{thm:minimal}
If $h$ is a minimal $\Delta_{\alpha,\beta}$-harmonic function on ${\mathsf{HT}}$, then (at least)
one of the following holds. 
\begin{enumerate}
 \item[\rm (a)] There is a minimal $\Delta^{{\mathbb H}}_{\alpha,\beta\mathsf{p}}$-harmonic 
function $h^{{\mathbb H}}$ on ${\mathbb H}$
such that $h(z,w) = h^{{\mathbb H}}(z)$ for all $\mathfrak{z} = (z,w) \in {\mathsf{HT}}$, or
\item[\rm (b)] there is a minimal $\Delta^{{\mathbb T}}_{\alpha,\beta}$-harmonic 
function $h^{{\mathbb T}}$ on ${\mathbb T}$
such that\\
$h(z,w) = h^{{\mathbb T}}(w)$ for all $\mathfrak{z} = (z,w) \in {\mathsf{HT}}$.
\end{enumerate}
\end{thm}

\begin{proof}
We know that $h=k(\cdot,\boldsymbol{\xi})$ for some $\boldsymbol{\xi} \in \mathcal{M}$, the Martin boundary
of ${\mathsf{HT}}$ with respect to $\Delta_{\alpha,\beta}$. There must be a sequence of 
points $\mathfrak{z}_n = (z_n=x_n+\mathfrak{i}\, y_n,w_n) \in {\mathsf{HT}}$ such that $\mathsf{d}_{{\mathsf{HT}}}(\mathfrak{o},\mathfrak{z}_n) \to \infty$ and
$$
h(\mathfrak{z}) = \lim_{n \to\infty} k(\mathfrak{z},\mathfrak{z}_n) \quad \text{for every}\quad
\mathfrak{z} \in {\mathsf{HT}}\,.
$$
Using compactness, we may suppose that $(\mathfrak{z}_n)$ converges to a boundary point in
$\widehat {\mathsf{HT}}$. Then we either have that $w_n \to \varpi$ in $\widehat{\mathbb T}$,
or else $y_n$ is bounded below by a positive constant and $z_n \to \boldsymbol{\infty} \in \partial{\mathbb H}$
in the topology of $\widehat{\mathbb H}$, that is, $y_n \to +\infty$ or $y_n \to y_0 \in \mathbb{R}$ and
$|x_n| \to +\infty$.
\\[6pt]
\emph{Case 1.} $w_n \to \varpi$. Let $\mathfrak{z} = (z,w)$ and $\mathfrak{z}' = (z,w') \in {\mathsf{HT}}$ be
two distinct points with $\pi_{{\mathbb H}}(\mathfrak{z})  =  \pi_{{\mathbb H}}(\mathfrak{z}') = z$, so that
$\mathfrak{h}(w)=\mathfrak{h}(w')$ in ${\mathbb T}$. Consider the confluent $v= w \curlywedge w' \in V({\mathbb T})$.
Since $w_n \to \varpi$, we must have for all but finitely many $n$ that 
$w_n \in {\mathbb T} \setminus {\mathbb T}_v\,$, where ${\mathbb T}_v$ is the subtree ${\mathbb T}_v$ of ${\mathbb T}$ rooted
at $v$ (i.e., consisting of all $v' \in {\mathbb T}$ with $v' \curlywedge v =v$). There is an 
isometry of ${\mathbb T}_v$ that exchanges $w$ and $w'$ (and
fixes $v$). It extends to an element $\psi$ in the horoyclic subgroup
of $\operatorname{\sf Aff}({\mathbb T})$, that fixes every point in ${\mathbb T} \setminus {\mathbb T}_v$. In particular,
$\psi w_n = w_n$ for all but finitely many $n$. Therefore 
$$
G(\mathfrak{z}',\mathfrak{z}_n) = G(\mathfrak{g}_{\,\psi\,}\mathfrak{z},\mathfrak{g}_{\,\psi\,}\mathfrak{z}_n) = G(\mathfrak{z},\mathfrak{z}_n)\,,
\quad \text{whence}\quad k(\mathfrak{z}',\mathfrak{z}_n)=k(\mathfrak{z},\mathfrak{z}_n)
$$
for all but finitely many $n$. Thus, as $n \to \infty$, we find 
$h(\mathfrak{z}')=h(\mathfrak{z})$ for all $\mathfrak{z}, \mathfrak{z}'$ with $\pi_{{\mathbb H}}(\mathfrak{z})=\pi_{{\mathbb H}}(\mathfrak{z}')$.  
This means precisely that there is a function $f^{{\mathbb H}}$ on ${\mathbb H}$
such that $h(z,w) = h^{{\mathbb H}}(z)$ for all $\mathfrak{z} = (z,w) \in {\mathsf{HT}}$.
It is now straightforward that $h^{{\mathbb H}}$ must be minimal 
$\Delta^{{\mathbb H}}_{\alpha,\beta\mathsf{p}}$-harmonic.\\[6pt]
\emph{Case 2.} $z_n \to \boldsymbol{\infty} \in \partial{\mathbb H}$ and $\inf y_n = y_0 > 0$. 
Let $b \in \mathbb{R}$. By \eqref{eq:metric} and standard properties
of the hyperbolic metric,
$$
\begin{aligned}
\mathsf{d}_{{\mathsf{HT}}}(\mathfrak{g}_{\,b\,}\mathfrak{z}_n,\mathfrak{z}_n) 
&= \mathsf{d}_{{\mathbb H}}\bigl((b+x_n) + \mathfrak{i}\, y_n, x_n + \mathfrak{i}\, y_n \bigr)\\
&= \mathsf{d}_{{\mathbb H}}(b + \mathfrak{i}\, y_n, \mathfrak{i}\, y_n )
\le \mathsf{d}_{{\mathbb H}}(b + \mathfrak{i}\, y_0, \mathfrak{i}\, y_0 ) =: d_b\,.
\end{aligned}
$$
Let $C_{d_b}$ be the corresponding Harnack constant in Proposition \ref{pro:harnack-green}. 
Then, using that $G(\cdot,\cdot)$ is $\mathcal{A}$-invariant and that
$\phi_{\alpha,\beta}(\mathfrak{g}_{\,b\,}\mathfrak{z}_n) = \phi_{\alpha,\beta}(\mathfrak{z}_n)$
$$
\begin{aligned}
k(\mathfrak{g}_{\,b\,}\mathfrak{z},\mathfrak{z}_n) 
&= \frac{G(\mathfrak{g}_{\,b\,}\mathfrak{z},\mathfrak{z}_n)}{G(\mathfrak{g}_{\,b\,}\mathfrak{z},\mathfrak{g}_{\,b\,}\mathfrak{z}_n)} 
\frac{G(\mathfrak{g}_{\,b\,}\mathfrak{z},\mathfrak{g}_{\,b\,}\mathfrak{z}_n)}{G(o,\mathfrak{z}_n)}\\
&= \frac{\mathbf{g}(\mathfrak{g}_{\,b\,}\mathfrak{z},\mathfrak{z}_n)}{\mathbf{g}(\mathfrak{g}_{\,b\,}\mathfrak{z},\mathfrak{g}_{\,b\,}\mathfrak{z}_n)}\, k(\mathfrak{z},\mathfrak{z}_n)
\le C_{d_b} \, k(\mathfrak{z},\mathfrak{z}_n).
\end{aligned}
$$
Letting $n \to \infty$, we obtain 
$$
\mathfrak{t}_{b}h(\mathfrak{z}):= h(\mathfrak{g}_{\,b\,}\mathfrak{z}) \le C_{d_b} \, h(\mathfrak{z}) \quad 
\text{for all} \; \mathfrak{z} \in {\mathsf{HT}}\,.
$$
Now, along with $h$, also $\mathfrak{t}_{b} h$ is positive harmonic, and minimality of
$h$ implies that the function $\mathfrak{t}_{b}h/h$ is constant. 
On the other hand, write $\mathfrak{z}=(x+\mathfrak{i}\, y, w)$. If we let 
$y \to +\infty$ (and simultaneously $\mathfrak{h}(w) = \log_{\mathsf{q}} y \to +\infty$)
then $d(\mathfrak{g}_{\,b\,}\mathfrak{z},\mathfrak{z}) = \mathsf{d}_{{\mathbb H}}(b + \mathfrak{i}\, y, \mathfrak{i}\, y) \to 0$. Consequently, 
Proposition \ref{pro:harnack}(a) implies that $h(\mathfrak{g}_{\,b\,}\mathfrak{z})/f(\mathfrak{z}) \to 1$.
Therefore $\mathfrak{t}_{b}f = f$.
This holds for every $b \in \mathbb{R}$. As in Case 1, this is equivalent with
the fact that there is a function $h^{{\mathbb T}}$ on ${\mathbb T}$ such that
$h(z,w) = h^{{\mathbb T}}(w)$ for all $\mathfrak{z} = (z,w) \in {\mathsf{HT}}$.
It is again straightforward that $h^{{\mathbb T}}$ must be minimal 
$\Delta^{{\mathbb T}}_{\alpha,\beta}$-harmonic.
\end{proof}

\begin{rmk}\label{rmk:factordouble} 
If both cases (a) and (b) of the last theorem occur simultaneously, then this 
means that there is a minimal $\Delta_{\alpha,\beta\mathsf{p}}^{\mathbb{R}}$-harmonic function $h^{\mathbb{R}}$ on $\mathbb{R}$
such that $h(z,w) = h^{\mathbb{R}}\bigl(\mathfrak{h}(w)\bigr)$ for all $\mathfrak{z} = (z,w) \in {\mathsf{HT}}$.
In the above proof, this happens when $\mathfrak{z}_n \to (\boldsymbol{\infty},\varpi)$.
\end{rmk}

\begin{proof}[\bf Proof of Theorem \ref{thm:decompose}]
Let $\mathcal{M}_{\min}^{(1)}$ and $\mathcal{M}_{\min}^{(2)}$ be the sets of all minimal 
harmonic functions on ${\mathsf{HT}}$ which are as in Theorem \ref{thm:minimal}(a) and (b),
respectiveley. Since the topology of the Martin boundary is the one of uniform 
convergence on compact sets, both are easily seen to be 
Borel sets in $\mathcal{M}$. Also, 
$\mathcal{M}_{\min} \setminus \mathcal{M}_{\min}^{(1)} \subset \mathcal{M}_{\min}^{(2)}$. 
If $f$ is any positive harmonic function on ${\mathsf{HT}}$ with integral representation
\eqref{eq:PoisMart}, then we can set
$$
f^{{\mathbb H}} = \int_{\mathcal{M}_{\min}^{(1)}} k(\cdot,\boldsymbol{\xi})\,d\nu^h(\boldsymbol{\xi})\,.
\quad\text{and}\quad f^{{\mathbb T}} = \int_{\mathcal{M} \setminus \mathcal{M}_{\min}^{(1)}} k(\cdot,\boldsymbol{\xi})\,d\nu^h(\boldsymbol{\xi})
\,.
$$
Then $f^{{\mathbb H}}(z,w)$ depends only on $z$, while $f^{{\mathbb T}}(z,w)$ depends only on $w$,
and $f$ is their sum, as proposed.
\end{proof}

\begin{rmk}\label{rmk:minimal} In view of Theorem \ref{thm:restrict-HT} and Remark
\ref{rmk:Hb}, as well as Lemma \ref{lem:restrict-T} (which also applies to the metric
tree $\mathbb{R}$ with vertex set ${\mathbb Z}$), we have the following:
a function on any of our spaces ${\mathsf{HT}}$, ${\mathbb H}$, ${\mathbb T}$, or $\mathbb{R}$ is minimal harmonic for
the respective one of our Laplacians if and only its restriction to $\operatorname{\sf LT}$, $\operatorname{\sf L}\!{\mathbb H}$, $V({\mathbb T})$, or
${\mathbb Z}$ (respectively) is minimal harmonic for the respective induced random walk.  
\end{rmk}

As a first application of Theorem \ref{thm:decompose}, we can 
clarify when the weak Liouville property holds, that is, when all bounded harmonic functions
are constant. For the following, we recall from \cite[Thm. 5.1]{BSSW2} that
$$
\lim_{t \to \infty} \frac{1}{t} \,\mathsf{d}_{{\mathsf{HT}}}(X_t,X_0) = 
|\ell(\alpha,\beta)| \;\; \text{almost surely, where}\;\; 
\ell(\alpha,\beta) = \frac{\log \mathsf{q}}{\mathsf{E}(\tau)} \,\frac{\mathsf{a}-1}{\mathsf{a}+1},
$$
with $\tau$ and $\mathsf{a}$ given by \eqref{eq:tau} and \eqref{eq:rha}, respectively. 
In particular, $\ell(\alpha,\beta)=0$ if and only if $\beta \mathsf{p} \,\mathsf{q}^{\alpha-1} = 1$.

\begin{thm}\label{thm:Liouville-HT} 
Suppose that $\mathsf{p} \ge 2$. Then the Laplacian $\Delta_{\alpha,\beta}$
on ${\mathsf{HT}}$ has the weak Liouville property if and only if $\ell(\alpha,\beta)=0$.
\end{thm}

This follows from Theorem \ref{thm:decompose} in combination with the following. 

\begin{pro}\label{cor:liou-T-H}$\,$\\ \emph{(a)} $\Delta^{{\mathbb T}}_{\alpha,\beta}$ has the weak 
Liouville property on ${\mathbb T}$ if and only if $\ell(\alpha,\beta) \le 0$. 
\\[4pt]
\emph{(b)}
$\Delta^{{\mathbb H}}_{\alpha,\beta\mathsf{p}}$ has the weak Liouville property on ${\mathbb H}$
if and only if $\ell(\alpha,\beta) \ge 0$. 
\end{pro}

(Note the striking difference in the range of validity of the weak Liouville
property on ${\mathsf{HT}}(\mathsf{q},\mathsf{p})$ with $\mathsf{p} \ge 2$ as compared to ${\mathbb H}(\mathsf{q}) = {\mathsf{HT}}(\mathsf{q},1)\,$! 
Also, compare those results with {\sc Karlsson and Ledrappier}~\cite{KaLe1}, \cite{KaLe2}.)

\begin{proof} We note that the weak Liouville property is the same
as triviality of the Poisson boundary, or equivalently, that the constant
function $\mathbf{1}$ is minimal harmonic. 
\\[5pt]
(a) It follows from \cite{CKW} that the weak Liouville property holds 
for the random walk $(W_{\tau(n)})$ on ${\mathbb T}$ if and only if $\ell(\alpha,\beta) \le 0$.
Via Lemma \ref{lem:restrict-T}, this transfers to $\Delta^{{\mathbb T}}_{\alpha,\beta}\,$.
Statement (a) also follows from Proposition \ref{pro:min-T} below. 
\\[5pt] 
(b) Among the different possible approaches, we adapt the method of Theorem \ref{thm:minimal}.
Since ${\mathbb H}(\mathsf{q}) = {\mathsf{HT}}(\mathsf{q},1)$, everything that we have stated and proved for
harmonic functions and the Green kernel on ${\mathsf{HT}}$ also applies to ${\mathbb H}$. (The 
decomposition result for positive harmonic functions becomes of course trivial.) 
In particular, we have the Green kernel 
$G^{{\mathbb H}}(w,z) = \phi(\alpha,\beta\mathsf{p})(z) \, \mathfrak{g}^{{\mathbb H}}(w,z)$, where
$\phi(\alpha,\beta\mathsf{p})(x + \mathfrak{i}\, y) = (\beta\mathsf{p})^{\lfloor \log_\mathsf{q} y \rfloor}\,y^{\alpha}$.
We also have the associated Martin kernel $k^{{\mathbb H}}(w,z)$. 

A basic theorem in Martin boundary theory says that $k^{{\mathbb H}}(\cdot, Z_t)$ converges almost
surely to a minimal harmonic function. We now show that when $\ell(\alpha,\beta) \ge 0$ then 
with positive probability, there is a random sequence $(Z_{t(n)})$ such that
$k^{{\mathbb H}}(\cdot, Z_{t(n)}) \to \mathbf{1}$. Thus, the constant function $\mathbf{1}$ on ${\mathbb H}$ must be minimal 
harmonic.

To show what we claimed, in the case when $\ell(\alpha,\beta) > 0$ then the random walk
$(Y_{\tau(n)})$ has positive drift, so that it tends to $\infty$ almost surely.
When $\ell(\alpha,\beta) = 0$, it is recurrent: with probability $1$, it visits every
point in ${\mathbb Z}$. Thus it is unbounded and has a random subsequence that tends to 
$\infty$. In both cases, we have a random sequence $t(n)$ such that 
$\,\operatorname{\text{\sl \!Im}} Z_{t(n)} \to \infty$ almost surely. 

Let us abbreviate $z_n = Z_{t(n)}\,$, so that $k^{{\mathbb H}}(\cdot, z_n)$ tends to a minimal
harmonic function $h$, and $\operatorname{\text{\sl \!Im}} z_n \to \infty$. We can use exactly the same method as
in Case 2 of the proof of Theorem \ref{thm:minimal} and get that $h(z+b) = h(z)$ for
every $z \in {\mathbb H}$ and every $b \in \mathbb{R}$. Thus, $h$ only depends on $\operatorname{\text{\sl \!Re}} z$, which means
that it arises from lifting a minimal $\Delta^{\mathbb{R}}_{\alpha,\beta\mathsf{p}}$-harmonic function $\widetilde h$ 
to ${\mathbb H}$ via $h(z) = \widetilde h (\log_{\mathsf{q}} \operatorname{\text{\sl \!Im}} z)$. We can apply Lemma \ref{lem:restrict-T}
to the specific case where the tree is just ${\mathbb R}$, with vertex set ${\mathbb Z}\,$: we see that the 
minimal $\Delta^{\mathbb{R}}_{\alpha,\beta\mathsf{p}}$ harmonic functions arise by interpolating the minimal
harmonic functions for the random walk $Y_{\tau(n)}$ on ${\mathbb Z}$. The latter minimal harmonic
functions are easily computed and are as follows.  

\smallskip

When $\ell(\alpha,\beta) = 0$ then that random walk is recurrent, so that all its positive harmonic
functions are constant. Thus, $h = \mathbf{1}$, as required.

\smallskip

When $\ell(\alpha,\beta) > 0$, the random walk has precisely two minimal harmonic 
functions, namely the constant function $\mathbf{1}$ and the function $\widehat f(k) = \mathsf{a}^{-k}$
(a well known exercise). We have to interpolate $\widehat f$ according to the variant of Lemma \ref{lem:restrict-T}
where $\mathsf{p} = 1$ and ${\mathbb T} = \mathbb{R}$ with vertex set ${\mathbb Z}$. This yields the unique 
$\Delta_{\alpha,\mathsf{p}\beta}^{\mathbb{R}}$-harmonic extension $\widehat h$. Now we have $h(z) = \widehat h (\log_{\mathsf{q}} \operatorname{\text{\sl \!Im}} z)$.

Thus, with probability one, $k^{{\mathbb H}}(\cdot, Z_t)$ converges either to the constant $\mathbf{1}$ or
to the latter function $h$. Suppose that it converges to $h$ with probability $1$. But then,
the Poisson boundary has to consist of only one point, and in this case, every bounded
harmonic function ought to be a multiple of $h$ - a contradiction since $h$ is unbounded.
Thus, with positive probability, the limit is the function $\mathbf{1}$ which thus must
be minimal harmonic on ${\mathbb H}$. (A posterori, the limit function must be $\mathbf{1}$ almost surely.)
\end{proof}

Theorem \ref{thm:decompose} tells us that for describing all positive
harmonic functions on ${\mathsf{HT}}$, we need to know all such functions
on ${\mathbb T}$ and on ${\mathbb H}$. Here, ``describing'' means that we
need to determine all minimal harmonic functions on each of those two spaces.
This task is very easy on ${\mathbb T}$. In view of Remark \ref{rmk:minimal}, we first
consider the induced random walk on $V({\mathbb T})$. The minimal harmonic functions 
for any transient nearest neighbour random walk  on 
(the vertex set of) a tree are completely
understood since the seminal article of {\sc Cartier}~\cite{Car}. For  $(W_{\tau(n)})$
on $V({\mathbb T})$, we have the following explicit formulas, taken from 
\cite[p. 424--425]{Wlamp}.

\begin{pro}\label{pro:min-T}
The minimal Martin boundary of the random walk on $V({\mathbb T})$ coincides with the whole Martin 
boundary, which is
$\partial{\mathbb T}$. For the transition probabilities of \eqref{eq:RWtree}, and with $\mathsf{a}$
as in \eqref{eq:rha}, the Martin kernels on $V({\mathbb T})$ are given as follows, where we  set
$$
\mathsf{b} =  \max\{\mathsf{a},1\} \quad\text{and}\quad \mathsf{c} = \min \{ \mathsf{a},1/\mathsf{a}\}/\mathsf{p}\,.
$$
For the boundary point $\varpi$,
$$
k^{{\mathbb T}}(v,\varpi) = \mathsf{b}^{-\mathfrak{h}(v)}\,, \quad v \in V({\mathbb T})\,.
$$
For $\xi \in \partial^*{\mathbb T}$, 
$$
k^{{\mathbb T}}(v,\xi) = \mathsf{b}^{-\mathfrak{h}(v)}\, \, \mathsf{c}^{\mathfrak{h}(o \curlywedge \xi) - \mathfrak{h}(v \curlywedge\xi)}\,, \quad v \in V({\mathbb T})\,.
$$
The minimal harmonic functions on the metric tree ${\mathbb T}$ are the respective extensions of
these functions according to \eqref{eq:extend-T}, which are also denoted
$k^{{\mathbb T}}(w,\xi)$, where $w \in {\mathbb T}$ and $\xi \in \partial{\mathbb T}$.
\end{pro}

We see indeed (as already mentioned) that the constant function $\mathbf{1}$ on ${\mathbb T}$ is minimal 
harmonic if and only if $\mathsf{a} \le 1$, that is $\ell(\alpha,\beta) \le 0$. 

\smallskip

Describing the minimal harmonic functions on ${\mathbb H}$ is in general more complicated and 
less explicit. There is one exception which we describe next. Namely, if $\beta\mathsf{p} = 1$
for the Laplacian on ${\mathsf{HT}}$, then its projection 
$\Delta_{\alpha,\beta\mathsf{p}}^{{\mathbb H}} = \Delta_{\alpha}^{{\mathbb H}}$
is nothing but the ordinary hyperbolic Laplacian on upper half plane ${\mathbb H}$ with vertical
drift parameter $\alpha$, as in the second line of Definition \ref{def:D2}. The bifurcation 
lines disappear in its analysis, and the minimal harmonic functions on ${\mathbb H}$ are well 
known; see e.g. {\sc Helgason} \cite{He}, or many other sources.

\begin{pro}\label{pro:min-H-noslice}
The minimal Martin boundary of $\Delta_{\alpha}^{{\mathbb H}} =\Delta_{\alpha,1}^{{\mathbb H}}$ on ${\mathbb H}$ 
coincides with the whole Martin boundary, which is $\partial{\mathbb H}$. 
The associated Martin kernels on ${\mathbb H}$ are given by the following
extended Poisson kernels.
For the boundary point $\boldsymbol{\infty}$,
$$
P_{\alpha}\bigl((x+\mathfrak{i}\, y), \boldsymbol{\infty}\bigr) = y^{\max\{ 1-\alpha,0\}}\,.
$$
For any boundary point $\zeta \in \mathbb{R}$, 
$$
P_{\alpha}\bigl((x+\mathfrak{i}\, y), \zeta\bigr) = y^{\max\{ 1-\alpha,0\}}
\left(\frac{\zeta^2 +1}{(\zeta-x)^2 + y^2}\right)^{\max\{\frac{\alpha}{2}, 1 - \frac{\alpha}{2}\}}.
$$
\end{pro}

At last, we can deduce the following result.

\begin{thm}\label{thm:allminimal} Consider ${\mathsf{HT}}(\mathsf{q},\mathsf{p})$ with $\mathsf{p} \ge 2$
and the Laplacian $\Delta_{\alpha,\beta}$ with $\beta=1/\mathsf{p}$, so that $\mathsf{a} = \mathsf{q}^{\alpha-1}$.
\\[4pt]
\emph{(I)} If $\alpha \ne 1$ then the minimal harmonic functions
on ${\mathsf{HT}}$ are parametrised by $\mathbb{R} \cup \partial^*{\mathbb T}$ and are given by
$$
\mathfrak{z} = (z,w) \mapsto k^{{\mathbb T}}(w,\xi)\,, \; \xi \in \partial^*{\mathbb T}\,,\quad\text{and}\quad
\mathfrak{z} = (z,w) \mapsto P_{\alpha}(z,\zeta)\,, \; \zeta \in \mathbb{R}\,.
$$
\emph{(II)} If $\alpha = 1$ then the minimal harmonic functions are as in \emph{(I)} plus,
in addition, the constant function $\mathbf{1}$.
\end{thm}
\begin{proof}
We first show that for $\zeta \in \mathbb{R}$, the function 
$\mathfrak{z} = (z,w) \mapsto P_{\alpha}(z,\zeta)$ is minimal harmonic. Suppose that 
$P_{\alpha}(z,\zeta) \ge h(z,w)$ for all $(z,w) = \mathfrak{z} \in {\mathsf{HT}}$, where $h$ is non-negative
harmonic on ${\mathsf{HT}}$. 
We decompose $h(z,w) = h^{{\mathbb H}}(z) + h^{{\mathbb T}}(w)$ according to Theorem \ref{thm:decompose}.
By minimality of $P_{\alpha}(\cdot,\zeta)$ on ${\mathbb H}$, we have $h^{{\mathbb H}} = c\cdot P_{\alpha}(\cdot,\zeta)$
for some $c \in [0\,,\,1]$. If $c=1$ then we are done. If $c < 1$, 
$$
P_{\alpha}(z,\zeta) \ge \frac{1}{1-c} h^{{\mathbb T}}(w) = \int_{\partial {\mathbb T}} k^{{\mathbb T}}(w,\xi)\,d\nu(\xi)
\quad\text{for all }\; (z,w) \in {\mathsf{HT}}\,,
$$
where $\nu$ is a Borel measure on $\partial {\mathbb T}$. We set $w = o$ and $z = x + \mathfrak{i}\,$, where
$x \in \mathbb{R}$. Then $P_{\alpha}(x+\mathfrak{i}\,,\zeta) \ge \nu(\partial T)$ for all $x$. When
$x \to \infty$, the Poisson kernel tends to $0$, so that the measure $\nu$ vanishes.
Thus, $h^{{\mathbb T}} \equiv 0$ and $h$ is a multiple of our Poisson kernel.

\smallskip

Next, we need that for $\xi \in \partial^*T$, the function 
$\mathfrak{z} = (z,w) \mapsto k^{{\mathbb T}}(w,\xi)$ is minimal harmonic.
This is completely analogous to the above, exchanging the
roles of ${\mathbb T}$ and ${\mathbb H}$.

\smallskip

When $\alpha = 1$, we know from Theorem \ref{thm:Liouville-HT} that the constant
function $\mathbf{1}$ is also minimal harmonic. 
When $\alpha \ne 1$, we still need to show that the two functions
$(z,w) \mapsto k^{{\mathbb T}}(w,\varpi)$ and $(z,w) \mapsto P_{\alpha}(z, \boldsymbol{\infty})$
are \emph{not} minimal harmonic.
These are the functions $\mathbf{1}$ and $(x+\mathfrak{i}\, y,w) \mapsto y^{1-\alpha}$.
The first one is not minimal, because there are non-constant
bounded harmonic functions. The second one is not minimal because
when $\alpha > 1$ it is already non-minimal on ${\mathbb H}$, while when
$\alpha < 1$ is is already non-minimal on ${\mathbb T}$.
\end{proof}

{\bf Outlook.} For a similar description of all minimal harmonic functions in the case when
$\bar \beta= \beta \mathsf{p} \ne 1$, one needs to study the boundary theory of ``sliced'' Laplacians
$\Delta^{{\mathbb H}}(\alpha,\bar\beta)$. A natural approach is to apply Theorem \ref{thm:restrict-HT}
to that case. In view of Remark \ref{rmk:Hb}, this means that one wants to determine
the minimal harmonic functions, or even the entire Martin compactification, for 
the random walk on $\operatorname{\sf Aff}({\mathbb H},\mathsf{q})$ with law $\widetilde\mu$. Since $\widetilde \mu$ has 
super-exponential moments
and continuous density with respect to Haar measure on that group, the guiding results are
those of {\sc Elie}~\cite{El1}, \cite{El2}. 
%
%
However, at a closer look, one finds that Elie's methods do not apply directly to our disconnected
subgroup of the affine group. A second possible approach is to clarify how the methods
of {\sc Ancona}~\cite{Anc} apply to $\Delta^{{\mathbb H}}(\alpha,\bar\beta)$ in the ``coercive'' case
$\bar\beta\,\mathsf{q}^{\alpha-1} \ne 1$, while in the case $\bar\beta\,\mathsf{q}^{\alpha-1} = 1$ one
might attempt to adapt the methods 
of {\sc Gou\"ezel}~\cite{Gou1}, \cite{Gou2} concerning Martin boundary of 
hyperbolic groups to our non-discrete
situation. Thus, additional work is required  -- this will be done separately.

\smallskip

{\bf Acknowledgement.} 
We thank Sara Brofferio for helpful interaction.

\end{document}